\documentclass[11pt, a4paper]{amsart}
\usepackage{amscd}
\usepackage{amsrefs}
\usepackage[mathscr]{eucal}
\usepackage[all]{xy}
\def\cal{\mathcal}

\makeatother
\newtheorem{Theorem}[equation]{Theorem}
\newtheorem{Corollary}[equation]{Corollary}
\newtheorem{Lemma}[equation]{Lemma}
\newtheorem{Proposition}[equation]{Proposition}

\theoremstyle{definition}
\newtheorem{Definition}[equation]{Definition}

\makeatletter
\renewcommand\section{\@startsection{section}{1}%
  {\z@}{.7\linespacing\@plus\linespacing}{.5\linespacing}%
  {\reset@font\normalfont\bfseries\centering}}
\makeatother

\theoremstyle{remark}
\newtheorem{Remark}[equation]{Remark}

\newtheorem*{Acknowledgements}{Acknowledgements}
\numberwithin{equation}{section}
\numberwithin{figure}{section}

\newcommand{\thmref}[1]{Theorem~\ref{#1}}
\newcommand{\secref}[1]{\S\ref{#1}}
\newcommand{\lemref}[1]{Lemma~\ref{#1}}
\newcommand{\propref}[1]{Proposition~\ref{#1}}

\newcommand{\corref}[1]{Corollary~\ref{#1}}

\newcommand{\subsecref}[1]{\S\ref{#1}}

\newcommand{\Romnum}[1]{\expandafter\uppercase\expandafter{\romannumeral #1}} 
\newcommand{\Z}{\mathbb{Z}}

\newcommand{\R}{\mathbb{R}}
\newcommand{\C}{\mathbb{C}}
\newcommand{\HH}{\mathbb{H}}

\newcommand{\CP}{\operatorname{\C P}}
\newcommand{\RP}{\operatorname{\R P}}

\newcommand{\SO}{\operatorname{\rm SO}}
\newcommand{\OO}{\operatorname{\rm O}}
\newcommand{\U}{\operatorname{\rm U}}
\newcommand{\SU}{\operatorname{\rm SU}}
\newcommand{\Pin}{\mathrm{Pin}}
\newcommand{\Spin}{\mathrm{Spin}}
\newcommand{\Spinc}{\Spin^{c}}
\newcommand{\Spincm}{\Spin^{c_-}}
\newcommand{\SP}{\operatorname{\rm Sp}}
\DeclareMathOperator{\rank}{rank}
\DeclareMathOperator{\Hom}{Hom}
\DeclareMathOperator{\End}{End}

\DeclareMathOperator{\Ima}{Im}
\DeclareMathOperator{\im}{im}

\DeclareMathOperator{\coker}{coker}
\DeclareMathOperator{\ad}{ad}
\DeclareMathOperator{\sign}{sign}
\DeclareMathOperator{\ind}{ind}

\newcommand{\M}{{\cal M}} 
\newcommand{\G}{{\cal G}} 

\newcommand{\id}{\operatorname{id}}

\DeclareMathOperator{\SW}{SW}
\newcommand{\X}{\hat{X}}
\newcommand{\hU}{\hat{\U}}
\newcommand{\deux}{\{\pm 1\}}
\newcommand{\Os}{\mathscr{O}}
\begin{document}
\title[Real structures and the $\Pin^-(2)$-monopole equations]{Real structures and the $\Pin^-(2)$-monopole equations}
%
%
%
%
%
\author{Nobuhiro Nakamura}
\address{Department of Mathematics, Division of Liberal Arts,
Osaka Medical College,
2-7 Daigaku-machi, Takatsuki City, Osaka, 569-8686, Japan}
\email{mat002@osaka-med.ac.jp}
%
%
%
\begin{abstract}
We investigate  the $\Pin^-(2)$-monopole invariants of symplectic $4$-manifolds and K\"{a}hler surfaces with real structures. 
We prove a nonvanishing theorem for real symplectic $4$-manifolds which is an analogue of Taubes' nonvanishing theorem of the Seiberg--Witten invariants for symplectic $4$-manifolds. 
Furthermore, the Kobayashi--Hitchin type correspondence for real K\"{a}hler surfaces is given. 
\end{abstract}
%
%
%
%
\maketitle
%
%
%
%
\section{Introduction}\label{sec:intro}
%
%
In the study of the Seiberg--Witten invariants, the computations of the invariants of K\"{a}hler surfaces are fundamental \cite{FM1,FM2,Brussee,Morgan, Teleman}, and they are based on a certain type of Kobayashi--Hitchin correspondence. 
On the other hand, Taubes'  works~\cite{Taubes1, Taubes2, Taubes3, Taubes4} on the Seiberg--Witten theory on symplectic $4$-manifolds begin with the nontriviality theorem~\cite{Taubes1} for the canonical $\Spinc$ structure.

The purpose of  this paper is to prove the theorems for $\Pin^-(2)$-monopole invariants~\cite{N2} parallel to the above results.
The $\Pin^-(2)$-monopole equations \cite{N1,N2} are a variant of the Seiberg--Witten equations twisted along a local system associated with a double covering $X\to\X$.
In fact,  an anti-linear involution $I$ on a $\Spinc$ structure $\mathfrak{s}$ on the double cover $X$ is defined, and the $\Pin^-(2)$-monopole theory on $\X$ can be identified with the $I$-invariant part of the Seiberg--Witten theory on $(X,\mathfrak{s})$.
We study K\"{a}hler surfaces and symplectic $4$-manifolds with {\it real structure}.
Our starting point is to understand the aforementioned $I$-action through the real structure.

Let us state our results more precisely.
Let $(X,\omega,\iota)$ be a closed {\it real symplectic $4$-manifold}, where $X$ is a closed smooth $4$-manifold $X$, $\omega$ is a symplectic form  and $\iota$ is an involution on $X$ such that $\iota^*\omega=-\omega$.
Let $J$ be a compatible almost complex structure such that $\iota_*\circ J=-J\circ\iota_*$.
Let $K$ be the canonical complex line bundle associated with $J$.  
We assume $(X,\omega,\iota)$ has {\it empty real part}, that is, the involution $\iota$ is {\it free}.
Let $\X$ be the quotient manifold $X/\iota$ and $\pi\colon X\to \X$ the projection.
Since $\iota$ induces an anti-linear involution on $K$, the quotient bundle $\hat{K}:=K/\iota$ is a nonorientable $\R^2$ bundle.
Let $\ell=X\times_{\deux}\Z$ be the local system associated to the double cover $X\to \X$.
Let $\ell_{\R}:=\ell\otimes\R =\det \hat{K}$.

The $\Pin^-(2)$-monopole equations are defined on a $\Spincm$ structure, which is a $\Pin^-(2)$-analogue of $\Spinc$ structure.
In the above situation, we define the {\it canonical $\Spincm$ structure} $\hat{\mathfrak{s}}_0$  associated with the real structure (\subsecref{subsec:canonical}).
The following theorem is an analogue of Taubes' nonvanishing theorem \cite{Taubes1}.
\begin{Theorem}\label{thm:canonical}
Suppose 
\begin{enumerate}  
\item $w_2(\X)+w_2(\hat{K})+ w_1(\ell_{\R})^2 =0$, 
\item $\pi^*\colon H^1(\X;\Z_2)\to H^1(X;\Z_2)$ is surjective.
\end{enumerate} 
Then there exists a unique canonical $\Spincm$ structure $\hat{\mathfrak{s}}_0$ on $X\to \X$. 
Furthermore, suppose $b_+^{\ell} = \dim H^+(\X;\ell_{\R})\geq 2$.
Then the $\Pin^-(2)$-monopole invariant  $\SW^{\Pin}(\X,\hat{\mathfrak{s}}_0)$ is $\pm 1$. 
\end{Theorem}
\begin{Remark}
We refer the readers to \cite{N1,N2} for the generality of the $\Pin^-(2)$-monopole theory.
In general, the $\Pin^-(2)$-monopole invariant is a $\Z_2$-valued invariant.  
We can define its $\Z$-valued refinement in some cases, e.g.,  when the moduli space is $0$-dimensional and orientable.
The statement that $\SW^{\Pin}(\X,\hat{\mathfrak{s}}_0)=\pm 1$ means that the invariant is nonzero anyway, and if its $\Z$-valued refinement is defined, then its value is $\pm1$.
\end{Remark}
Like the ordinary Seiberg--Witten theory, the $\Pin^-(2)$-monopole theory has a symmetry of  conjugation (\subsecref{subsec:symm}).
The conjugate of the canonical $\Spincm$ structure $\hat{\mathfrak{s}}_0$ is the {\it anti-canonical $\Spincm$ structure} $\hat{\mathfrak{s}}_0\hat\otimes\hat{K}$ which is obtained by twisting $\hat{\mathfrak{s}}_0$ by $\hat{K}$.
\thmref{thm:canonical} with \corref{cor:conj} implies the following.
\begin{Corollary}\label{cor:anticanonical}
$\SW^{\Pin}(\X,\hat{\mathfrak{s}}_0\hat\otimes\hat{K} )= \pm 1$.
\end{Corollary}

An $\OO(2)$ bundle $\hat{L}$  called {\it characteristic bundle} is associated with a $\Spincm$ structure.
Let $\tilde{c}_1(\hat{L})\in H^2(\X;\ell)$ be its $\ell$-coefficient Euler class.
Since $\iota^*\omega = -\omega$, there is a $\ell_{\R}$-valued self-dual closed $2$-form $\hat{\omega}\in \Omega^2(X;\ell_{\R})$ such that $\pi^*\hat{\omega}=\omega$.
The following is an analogue of \cite[Theorem 2]{Taubes2}.
\begin{Theorem}\label{thm:more}
Under the assumptions of \thmref{thm:canonical},
if the $\Pin^-(2)$-monopole invariant for a $\Spincm$ structure $\hat{\mathfrak{s}}$ on $X\to \X$  is nonzero, then its characteristic bundle $\hat{L}$ satisfies
\begin{equation}\label{eq:more}
\left|\tilde{c}_1(\hat{L})\cdot[\hat{\omega}]\right|\leq \tilde{c}_1(\hat{K})\cdot[\hat{\omega}],
\end{equation}
and the virtual dimension $d(\hat{\mathfrak{s}})$ of the moduli space is $0$.
\end{Theorem}

Suppose further that $(X,\omega)$ is a compact K\"{a}hler surface and $\iota$ is an antiholomorphic free involution.
In such a case, we prove a certain kind of Kobayashi--Hitchin correspondence (\secref{sec:realKL}).
In fact, the $\Pin^-(2)$-monopole moduli space for $\X$ can be identified with the $I$-invariant part of the space of simple holomorphic pairs of holomorphic structures on a line bundle with nonzero holomorphic sections.
(Furthermore, a simple holomorphic pair is identified with an effective divisor on $X$.)
By using this description, we can compute the $\Pin^-(2)$-monopole invariants for the quotient manifolds of several kinds of K\"{a}hler surfaces.
The following is an analogue of \cite[Theorem 7.4.1]{Morgan}.

\begin{Theorem}\label{thm:general}
Let $X$ be a minimal K\"{a}hler surface of general type. 
Suppose $\iota\colon X\to X$ is an anti-holomorphic involution without fixed points satisfying the assumptions in \thmref{thm:canonical}. 
Then 
\[
\SW^{\Pin}(\X,\hat{\mathfrak{s}})=\left\{\begin{aligned}
\pm 1 \quad & \hat{\mathfrak{s}}=\hat{\mathfrak{s}}_0\text{ or } \hat{\mathfrak{s}}_0\hat\otimes\hat{K}\\
0\ \quad & \text{ otherwise }
\end{aligned}\right.
\]
\end{Theorem}
%
%
The hypersurfaces in $\CP^3$ with complex conjugation satisfy the conditions of \thmref{thm:general} (\subsecref{subsec:general}).
We also compute the invariants of elliptic surfaces in \subsecref{subsec:elliptic}.

\begin{Acknowledgements}
The author would like to thank the referee for his comments on the earlier version of the paper.
  The  author is supported in part by JSPS Grant-in-Aid for Scientific Research (C) 25400096 and 19K03506.
\end{Acknowledgements}

%
%
\section{$\Spincm$ structures induced from the real structure}\label{sec:2}
%
%
\subsection{Reduction of the frame bundle}
Recall the isomorphism $\U(2)\cong(\U(1)\times\SU(2))/\deux$.
Define the group $\hU(2)$ by
\[
\hU(2) =  (\Pin^-(2)\times\SU(2))/\deux.
\]
Then $\hU(2)/\Pin^-(2)=\SO(3)$, $\hU(2)/\SU(2)=\OO(2)$, the identity component of $\hU(2)$ is $\U(2)$, and $\hU(2)/\U(2)=\deux$.
We have an exact sequence
\begin{equation}\label{eq:hom}
1\to\deux\to\hU(2)\overset{\sigma}{\to} \OO(2)\times\SO(3)\to 1.
\end{equation}
Note that $\hU(2)$ is embedded in $\SO(4)$ as
\begin{equation}\label{eq:emb}
\hU(2) =  \frac{\Pin^-(2)\times\SU(2)}{\deux}\subset \frac{\SU(2)\times\SU(2)}{\deux}=\SO(4).
\end{equation}
Suppose we have a manifold $\hat{Y}$ with a double covering $Y\to \hat{Y}$ and a principal $\hU(2)$-bundle $P$ over $\hat{Y}$ such that $P/\U(2)\cong Y$.
Then we have an $\OO(2)$-bundle $P_O:=P/\SU(2)$ such that $P_O/\SO(2)\cong Y$ and an $\SO(3)$-bundle $P_S:=P/\Pin^-(2)$. 
Conversely, the following holds
\begin{Proposition}\label{prop:F}
For a double covering $Y\to \hat{Y}$, let $\ell_{\R}=Y\times_{\deux}\R$ and suppose an $\OO(2)$-bundle $P_O$ such that $P_O/\SO(2)\cong Y$ and an $\SO(3)$-bundle $P_S$ are given.
If $w_2(P_O)+w_1(\ell_{\R})^2 =w_2(P_S)$, then there exists a $\hU(2)$-bundle $P$ such that 
\[
P/\Pin^-(2)\cong P_S,\quad P/\SU(2)\cong P_O,\quad P/\U(2)\cong Y.
\] 
\end{Proposition}
\begin{proof}
({\it Cf.}~\cite[Proposition 11]{N1}.)
Note that the image of $\Pin^-(2)\subset \SP(1)=\Spin(3)$ by the canonical homomorphism $\Spin(3)\to \SO(3)$ is a copy of $\OO(2)$ embedded in $\SO(3)$.
The embedding $\OO(2)\subset\SO(3)$ is given by $A\mapsto A\oplus \det A$.
Embed $\OO(2)\times \SO(3)$ into $\SO(6)$ by using this embedding.
Then we have a commutative diagram
\[
\begin{CD}
1@>>>\deux @>>>\hU(2)@>>>\OO(2)\times\SO(3)@>>>1\\
@.@|@VVV @VVV @. @.\\
1@>>>\{\pm 1\}@>>>\Spin(6)@>>>\SO(6)@>>>1.
\end{CD}
\]
The diagram leads to a commutative diagram of fibrations
$$
\begin{CD}
K(\Z_2,1)@>>>B\hU(2)@>>>B\OO(2)\times B\SO(3)@>>>K(\Z_2, 2)\\
@VVV @VVV @VVV @VVV\\
K(\Z_2,1)@>>>B\Spin(6)@>>>B\SO(6)@>{w_2}>>K(\Z_2, 2).
\end{CD}
$$
From these, we see that 
$$w_2(P_S\oplus P_O\oplus\det P_O)=w_2(P_S)+ w_2(P_O)+w_1(\ell_{\R})^2=0$$  
is the required condition.
\end{proof}
\begin{Remark}
The choice of $P$ is not unique. 
The possibility of $P$ is parametrized  by $H^1(\hat{Y};\Z_2)$.
\end{Remark}

Recall the embedding 
\[
\U(2)=(\U(1)\times\SU(2))/\deux\subset (\SU(2)\times\SU(2))/\deux=\SO(4)
\]
and a commutative diagram
\[
\begin{CD}
1@>>>\deux @>>>\U(2)@>{\sigma^\prime}>>\U(1)\times\SO(3)@>>>1\\
@.@|@VVV @VVV @. @.\\
1@>>>\{\pm 1\}@>>>\SO(4)@>>>\SO(3)\times\SO(3)@>>>1.
\end{CD}
\]

Let $(X,\omega,J)$ be a symplectic $4$-manifold with compatible almost complex structure $J$.
Fixing a Hermitian metric on $TX$, we obtain a $\U(2)$ reduction $P_F$ of the $\SO(4)$-frame bundle.
Then we have a $\U(1)$-bundle $P_K:=P_F/\SU(2)$ and an $\SO(3)$-bundle $P_S:=P_F/\U(1)$. 
Let $K=\Lambda^{2,0}(X)$ and $K^{-1}=\Lambda^{0,2}(X)$ be respectively the canonical and anti-canonical line bundles associated with the almost complex structure $J$. 
Note that $\Lambda^+(X)\otimes_{\R} \C\cong \C\omega\oplus K\oplus K^{-1}$. 
Then we can identify
\[
P_K\times_{\U(1)}\C\cong K\cong K^{-1}
\]
as {\it real} vector bundles.
We assume $P_K\times_{\U(1)}\C =  K$.
On the other hand, 
\[
\Lambda^-(X)\cong P_S\times_{\SO(3)}\R^3.
\]
%

Let $(X,\omega,\iota)$ be a closed real symplectic $4$-manifold without real part.
Then $X$ admits an almost complex structure $J$ compatible to $\omega$ such that $\iota_*\circ J =-J\circ\iota_*$.
Fixing such a $J$,  we have a $\U(2)$ reduction $P_F$ of the $\SO(4)$-frame bundle.
Let $P_K$ and $P_S$ be the induced $\U(1)$ and $\SO(3)$ bundles.
Let $\X$ be  the quotient manifold $\X=X/\iota$ and $\ell_{\R}=X\times_{\deux}\R$. 
The involution $\iota$ induces a bundle automorphism $\tilde{\iota}$ of $P_S$ such that $\tilde{\iota}^2=1$, and its quotient bundle $\hat{P}_S=P_S/\tilde{\iota}$ over $\X$ has the property that
\[
\hat{P}_S\times_{\SO(3)}\R^3 = \Lambda^-(\X).
\]
On the other hand, $\iota$ does not induce a bundle automorphism on $P_K$ since $\iota$ is not complex linear.
However $\iota$ induces an anti-linear involution on the canonical bundle $K=P_K\times_{\U(1)}\C$.
Then the quotient bundle $\hat{K}=K/\iota$ is a nonorientable $\R^2$ bundle over $\X$ such that $\det\hat{K}=\ell_{\R}$.
Let $\hat{P}_K$ be the $\OO(2)$-bundle over $\X$ of orthogonal frames on $\hat{K}$.
By \propref{prop:F}, we have a $\hU(2)$-bundle $\hat{P}$ which induces $\hat{P}_S$ and $\hat{P}_K$ if 
$w_2(\X)=w_2(\hat{K})+ w_1(\ell_{\R})^2$.
Note that the $\deux$-bundle $\hat{P}/\U(2)\to \X$ is isomorphic to $\pi\colon X\to \X$.
Fix an isomorphism between them.
Then $\hat{P}\to \hat{P}/\U(2)$ can be considered as a $\U(2)$-bundle over $X$.
This $\U(2)$-bundle $\hat{P}\to \hat{P}/\U(2)= X$ is denoted by $P^\prime$.
\begin{Proposition}\label{prop:P}
Suppose 
\begin{enumerate}  
\item $w_2(\X)+w_2(\hat{K})+ w_1(\ell_{\R})^2 =0$, 
\item $\pi^*\colon H^1(\X;\Z_2)\to H^1(X;\Z_2)$ is surjective.
\end{enumerate} 
Then we can take a $\hU(2)$-bundle $\hat{P}\to \X$  such that 
\begin{equation}\label{eq:P}
\hat{P}/\Pin^-(2)\cong \hat{P}_S, \quad \hat{P}/\SU(2)\cong \hat{P}_K,\quad \hat{P}/\U(2)\cong X, \quad P^\prime\cong P_F.
\end{equation}
Furthermore $\hat{T}=\hat{P}\times_{\hU(2)}\R^4$ is isomorphic to $T\X$, where $\hat{T}$ is defined via the embedding \eqref{eq:emb}.
\end{Proposition}
\begin{proof}
By the proof of \propref{prop:F}, we see that the set of isomorphism classes of $\hU(2)$-bundle $\hat{P}$ which induces the same $\hat{P}_K$ and $\hat{P}_S$ is parametrized by $H^1(\X;\Z_2)$.
If a choice of $\hat{P}$ is given, then every other choice is obtained by tensoring a real line bundle.
Similarly, the set of isomorphism classes of $\U(2)$-bundle $P_F$ which induces the same $P_K$ and $P_S$ is parametrized by $H^1(X;\Z_2)$.
Now choose  $\hat{P}$ which induces $\hat{P}_K$ and $\hat{P}_S$.
Then it follows from the construction that the $\U(2)$-bundle $P^\prime$ induces $P_K$ and $P_S$.
Thus the difference between $P_F$ and $P^\prime$ is given by an element of $H^1(X;\Z_2)$.
Under the assumption, the difference can be annihilated by tensoring an appropriate real line bundle over $\X$ with $\hat{P}$.

Since $\pi^*\hat{P}\times_{\hU(2)}\R^4=P^\prime\times_{\U(2)}\R^4 = P_F\times_{\U(2)}\R^4=TX$, we have $\pi^*\hat{T}\cong TX\cong\pi^*T\X$.
From this, it follows that $e(\hat{T})=e(T\X)$ and $p_1(\hat{T})=p_1(T\X)$.
Consider the homomorphisms $\hU(2)\overset{\sigma}{\to}\OO(2)\times \SO(3)\overset{p}{\to}\SO(3)$ where $p$ is the projection to the second factor.
Then the composite map $p\circ\sigma\colon \hU(2)\to\SO(3)$ factors through $\hU(2)\hookrightarrow\SO(4)\to\SO(3)$.
Then we have a commutative diagram
\[
\xymatrix{
B\hU(2)\ar[r]\ar[rd] & B\SO(3)\\
&B\SO(4)\ar[u]
}
\]
From this, it follows that $w_2(\hat{T})=w_2(\hat{P}_S)=w_2(\X)$.
Therefore $\hat{T}\cong T\X$.
\end{proof}
\begin{Remark}
The choice of $\hat{P}$ is not unique. 
The possibility of $\hat{P}$ is parametrized  by $\ker(\pi^*\colon H^1(\X;\Z_2)\to H^1(X;\Z_2))$.
\end{Remark}

\subsection{Canonical $\Spincm$ structure}\label{subsec:canonical}

Recall that the canonical $\Spinc$ structure $\mathfrak{s}_0$ over $X$ with respect to the almost complex structure $J$ is defined from the $\U(2)$-reduction $P_F$, and it has the positive spinor bundle $W^+_0$ of the form $W^+_0 = \underline{\C}\oplus K^{-1}$. 
In this subsection, we define the canonical $\Spincm$ structure over $X\to\X$ induced from the real structure on $X$.

Recall that
\[
\Spincm(4)
=\frac{\SU(2)\times\SU(2)\times\Pin^-(2)}{\deux} = \frac{\SP(1)\times\SP(1)\times\Pin^-(2)}{\deux}.
\] 
A $\Spincm$ structure $\hat{\mathfrak{s}}$ on $X\to\X$ consists of a $\Spincm(4)$-bundle $Q$ over $\X$, an isomorphism of $\Z/2$-bundles $Q/\Spinc(4)\cong X$, and an isomorphism between the $\SO(4)$-frame bundle and $Q/\Pin^-(2)$. 
The $\OO(2)$-bundle $\hat{L}=Q/\Spin(4)$  is called the {\it characteristic bundle} of $\hat{\mathfrak{s}}$.  
The bundle  $\hat{L}$ has a $\ell$-coefficient orientation and its Euler class is denoted by $\tilde{c}_1(\hat{L})\in H^2(\X;\ell)$.  
We often make no distinction between $\hat{L}$ and its associated $\R^2$-bundle.
Let $\HH_{\pm}$ be the $\Spincm(4)$ modules which are copies of $\HH$ as vector spaces such that the action of $[q_+,q_-,u]\in\Spincm(4)=(\SP(1)\times\SP(1)\times\Pin^-(2))/\deux$  on $\phi\in\HH_{\pm}$ is given by $q_{\pm}\phi u^{-1}$. 
Then the associated bundles $W_{\pm}=Q\times_{\Spincm(4)}\HH_{\pm}$ are the spinor bundles of $\hat{\mathfrak{s}}$.

Note that the embedding $\hU(2)\hookrightarrow \SO(4)$ factors through another embedding  $\varepsilon\colon\hU(2)\to\Spincm(4)$ which is defined  by
\begin{align*}
\varepsilon\colon\hU(2) = \frac{\Pin^-(2)\times\SU(2)}{\deux} \to& \frac{\SU(2)\times\SU(2)\times\Pin^-(2)}{\deux} = \Spincm(4),\\
(u,q)\mapsto& (u,q,u). 
\end{align*}
When we have a $\hU(2)$-bundle  $\hat{P}$ as in \propref{prop:P}, a $\Spincm$ structure $\hat{\mathfrak{s}}$ over $X$  is defined via the embedding $\varepsilon$.
The $\Spincm(4)$-bundle $Q$ of $\hat{\mathfrak{s}}$ is given by
\[
Q=\hat{P}\times_{\hU(2)}\Spincm(4),
\]
Note that the characteristic $\OO(2)$-bundle of $\hat{\mathfrak{s}}$ is $\hat{P}_K$.
The positive spinor bundle $\hat{W}^+$ is defined by the adjoint action of $\Pin^-(2)$ on the space of quaternions $\HH = \C\oplus j\C$:
\[
W^+ = Q\times_{\Spincm(4)}\HH_+ = \hat{P}\times_{\Pin^-(2)}\HH.
\]
For $u\in\U(1)$ and $z\in\C$, the adjoint action is given by 
\begin{equation}\label{eq:ad}
\begin{aligned}
\ad_u(z) &=uzu^{-1} =z,\\ 
\ad_{ju}(z) &= juzu^{-1}j^{-1}=\bar{z},\\ 
\ad_u(jz)&= u^2jz =u^2\bar{z}j, \\
\ad_{ju}(jz)&=u^{-2}j\bar{z}=u^{-2}zj.
\end{aligned}
\end{equation}
This action preserves the components $\C$ and $j\C$.
It follows from \eqref{eq:ad} that $\hat{W}^+$  is decomposed into the direct sum of two $\R^2$ bundles as $\hat{W}^+=\hat{E}_1\oplus \hat{E}_2$ such that $\det \hat{E}_1 =\det \hat{E}_2=\ell_{\R}$.
Define the $\R^2$-bundle $\underline{\hat{\C}}$ by $\underline{\hat{\C}} = X\times_{\deux}\C$, where $\deux$ acts on $\C$ by complex conjugation. 
Note that $\underline{\hat{\C}} = \underline{\R}\oplus \ell_{\R}$.
Since $\pi^*\hat{W}^+ = W^+_0 = \underline{\C}\oplus K^{-1}$, we see that $\hat{W}^+$ has a form of
\[
\hat{W}^+ = (\underline{\hat{\C}}\oplus \hat{K}^{-1})\otimes \lambda^\prime,
\]
where $\hat{K}^{-1}=(K^{-1})/\iota$ (which is the characteristic bundle of the $\Spincm$ structure) and $\lambda^\prime$ is a real line bundle over $\X$ with $\pi^*\lambda^\prime$ trivial.
Note that tensoring $\lambda^\prime$ to $\hat{P}$ changes $\hat{W}^+$ into $\hat{W}^+_0 = \underline{\hat{\C}}\oplus  \hat{K}^{-1}$.
Now we define the canonical $\Spincm$ structure.
\begin{Definition}\label{def:canonical}
A $\Spincm$ structure $\mathfrak{s}_0$ on $X\to \X$ is {\it canonical}  if it is defined from a $\hU(2)$-bundle $\hat{P}$ satisfying \eqref{eq:P} and its  positive spinor bundle $\hat{W}^+_0$ has a form of 

\[
\hat{W}^+_0 = \underline{\hat{\C}}\oplus  \hat{K}^{-1}.
\]
\end{Definition}
The above discussion implies the following.
\begin{Corollary}\label{cor:canonical}
Suppose (1) and (2) in \thmref{thm:canonical}.
Then there exists a unique canonical $\Spincm$ structure on $X\to\X$.
\end{Corollary}
Recall $\R^2$-bundles $\hat{E}$ such that $\det \hat{E}=\ell_{\R}$ with $\ell_{\R}$-coefficient orientation are classified by $\tilde{c}_1(\hat{E})\in H^2(\X;\ell)$.
We call an $\R^2$-bundle $\hat{E}$ such that $\det \hat{E}=\ell_{\R}$ an {\it $\R^2$-bundle twisted along $\ell_{\R}$.}
For $\R^2$-bundles $\hat{E}_1$ and $\hat{E}_2$ twisted along $\ell_{\R}$, there exists another $\R^2$-bundle $\hat{E}$ twisted along $\ell_{\R}$ such that $\tilde{c}_1(\hat{E})=\tilde{c}_1(\hat{E}_1)+\tilde{c}_1(\hat{E}_2)$, which can be considered as a ``twisted tensor product" of $\hat{E}_1$ and $\hat{E}_2$.
We write $\hat{E}=\hat{E}_1\hat{\otimes} \hat{E}_2$.

If $X\to \X$ admits a $\Spincm$ structure, then the set of equivalence classes of $\Spincm$ structures is also parametrized by $H^2(X;\ell)$.
Once a $\Spincm$ structure is given, the other $\Spincm$ structures are given by ``tensoring" an $\R^2$-bundle $\hat{E}$ twisted along $\ell_{\R}$.
In fact, when we have a canonical $\Spincm$ structure $\hat{\mathfrak{s}}_0$, there is a $\Spincm$ structure whose positive spinor bundle is 
\[
\hat{W}=\hat{E}\oplus (\hat{E}\hat{\otimes} \hat{K}^{-1}).
\]  
This $\Spincm$ structure is denoted by $\hat{\mathfrak{s}}_0\hat{\otimes}\hat{E}$.
\begin{Definition}\label{def:anti-canonical}
The $\Spincm$ structure $\hat{\mathfrak{s}}_0\hat{\otimes}\hat{K}$ is called the {\it anti-canonical $\Spincm$ structure}.
This has the spinor bundle of the form
\[
\hat{W}=\hat{K}\oplus \underline{\hat{\C}}.
\]
\end{Definition}
\begin{Remark}
For a $\Spincm$ structure $\hat{\mathfrak{s}}$ over $\pi\colon X\to\X$, let $\pi^*\hat{\mathfrak{s}}$ be the $\Spinc$ structure  over $X$ which is the pull-back of  $\hat{\mathfrak{s}}$.
Then  $\pi^*\hat{\mathfrak{s}}$ has two $\Spinc$ reductions, and one of them is the {\it canonical reduction} \cite[\S2.4]{N2}.
Then it can be seen that the canonical reduction of the pull-back $\pi^*\hat{\mathfrak{s}}_0$ of the canonical $\Spincm$ structure $\hat{\mathfrak{s}}_0$ is the canonical $\Spinc$ structure $\mathfrak{s}_0$ on $X$, and the canonical reduction of $\pi^*(\hat{\mathfrak{s}}_0\hat\otimes \hat{K})$ is the anti-canonical $\Spinc$ structure $\mathfrak{s}_0\otimes K$.
\end{Remark}
\subsection{A symmetry in the $\Pin^-(2)$-monopole theory}\label{subsec:symm}
It is well-known that there is a symmetry of complex conjugation in the Seiberg--Witten theory \cite[\S6.8]{Morgan}.
We explain a similar symmetry in the $\Pin^-(2)$-monopole theory.
The conjugation of a quaternion $z\in\HH$ is given by
\[
z=a+ib+jc+kd\mapsto \bar{z}=a-ib-jc-kd.
\]
Define the conjugation  $\alpha\colon\Spincm(4)\to\Spincm(4)$ by 
\[
\alpha([q,z])=[q,\bar{z}] \text{ for } [q,z]\in\Spincm(4)=\Spin(4)\times_{\deux}\Pin^-(2).
\]
For a $\Spincm(4)$-bundle $P$, let $P^c$ be the $\Spincm(4)$-bundle such that the total space is same with $P$, but the action of $\Spincm(4)$ is given by $p\cdot \alpha(q)$ for $p\in P=P^c$ and $q\in \Spincm(4)$.

For a $\Spincm$ structure $\hat{\mathfrak{s}}$ with $\Spincm(4)$-bundle $P$, we have a  $\Spincm$ structure $\hat{\mathfrak{s}}^c$ whose $\Spincm(4)$-bundle is $P^c$.
We call $\hat{\mathfrak{s}}^c$ the {\it conjugate of $\hat{\mathfrak{s}}$}.

Recall that $\ell_{\R}$-oriented $\R^2$-bundles $\hat{E}$ twisted along $\ell_{\R}$ are classified by $\tilde{c}_1(\hat{E})$.
For such an $\hat{E}$, let $\hat{E}^c$ be an $\R^2$-bundle such that $\tilde{c}_1(\hat{E}^c)=-\tilde{c}_1(\hat{E})$.
We collect several facts on conjugate which can be easily seen.
\begin{Proposition}\label{prop:symm}
For a $\Spincm$ structure  $\hat{\mathfrak{s}}$ and its conjugate $\hat{\mathfrak{s}}^c$, we have the following:
\begin{enumerate} 
\item If $\hat{L}$ is the characteristic bundle for $\hat{\mathfrak{s}}$, then $\hat{L}^c$ can be identified with the characteristic bundle of $\hat{\mathfrak{s}}^c$. 
In particular,  $\tilde{c}_1(\hat{L}^c)=-\tilde{c}_1(\hat{L})$.

\item For  an $\R^2$-bundle $E$ twisted along $\ell_{\R}$,
\[
(\hat{\mathfrak{s}}\hat{\otimes}\hat{E})^c = \hat{\mathfrak{s}}^c\hat{\otimes}\hat{E}^c.
\]
\item The conjugate of the canonical $\Spincm$ structure  is the anti-canonical $\Spincm$ structure, i.e, $\hat{\mathfrak{s}}_0^c = \hat{\mathfrak{s}}_0\hat{\otimes}\hat{K}$.
\item If $\mathfrak{s}$ is the canonical reduction of $\pi^*\hat{\mathfrak{s}}$, then the canonical reduction of $\pi^*\hat{\mathfrak{s}}^c$ is the complex conjugate $\bar{\mathfrak{s}}$ of  $\mathfrak{s}$.
\end{enumerate}
\end{Proposition}
Let $\hat{\mathfrak{s}}$ be a $\Spincm$ structure on $\pi\colon X\to \X$ and $\mathfrak{s}$  the canonical reduction of $\pi^*\hat{\mathfrak{s}}$. 
By \cite[\S4.5]{N1} (see also \cite[\S2.5]{N2}), there is an involution $I$ on the Seiberg--Witten theory on $(X,\mathfrak{s})$, and a bijective correspondence between the $\Pin^-(2)$-monopole solutions on $(\hat{X},\hat{\mathfrak{s}})$ and the $I$-invariant Seiberg--Witten solutions on $(X,\mathfrak{s})$.
Let us recall the relation between the downstairs $(\hat{X},\hat{\mathfrak{s}})$ and upstairs  $(X,\mathfrak{s})$ more precisely.
Note that $\iota^*\mathfrak{s}$ is isomorphic to the complex conjugation  $\overline{\mathfrak{s}}$ of $\mathfrak{s}$.
For a configuration $(A,\phi)$ on $(X,\mathfrak{s})$, $I(A,\phi)$ is defined by
\[
I(A,\phi) = (\overline{\iota^*A},\overline{\iota^*\phi}),
\]
where $\overline{\cdot}$ means complex conjugation.
For a configuration $(\hat{A},\hat{\phi})$ on $(\hat{X},\hat{\mathfrak{s}})$, we have a unique configuration $(A,\phi)$ on $(X,\mathfrak{s})$ such that $\phi=\pi^*\hat\phi$ and $A$ is the canonical $\U(1)$ reduction of the $\OO(2)$-connection $\pi^*A$ which is the pull-back of $A$.
We call $(A,\phi)$ {\it the lift} of  $(\hat{A},\hat{\phi})$.
Note that the lift $(A,\phi)$ is $I$-invariant.

The gauge transformation group of the $\Pin^-(2)$-monopole theory is given by
\[
\hat{\G} = \Gamma (X\times_{\deux}\U(1)),
\]
where $\deux$ acts on $\U(1)$ by $u\mapsto u^{-1}$.
Then $\hat{\G}$ can be identified with the $I$-invariant gauge transformation group on the upstairs $X$.
That is, the $I$-action on $\G=C^\infty(X,\U(1))$ is given by $f\mapsto \overline{\iota^*f}$, and we have a natural identification $\hat{\G}=\G^I$.

The $\Pin^-(2)$-monopole moduli space is 
\[
\hat{\M}(\X,\hat{\mathfrak{s}})=\{\text{ $\Pin^-(2)$-monopole solutions on $\hat{\mathfrak{s}}$ } \}/\hat{\G}.
\]
 and this is identified with the $I$-invariant moduli space,
\[
\M(X,\mathfrak{s})^I=\{\text{ Seiberg--Witten solutions on $\mathfrak{s}$ } \}^I/\G^I.
\]
By \propref{prop:symm}, we have the identifications,
\[
\hat{\M}(\X,\hat{\mathfrak{s}})\cong \M(X,\mathfrak{s})^I \cong \M(X,\overline{\mathfrak{s}})^I \cong \hat{\M}(\X,\hat{\mathfrak{s}}^c).
\]
The second identification is the isomorphism of complex conjugation in the ordinary Seiberg--Witten theory.
\begin{Corollary}\label{cor:conj}
$\SW^{\Pin}(\X,\hat{\mathfrak{s}}^c)=\pm \SW^{\Pin}(\X,\hat{\mathfrak{s}})$.
\end{Corollary}
%
%
\section{Real symplectic $4$-manifolds}\label{sec:proofs}
%
%

In this section, we prove \thmref{thm:canonical} and \thmref{thm:more}.
Suppose a closed real symplectic $4$-manifold $(X,\omega,\iota)$ satisfies the assumption of \thmref{thm:canonical}.

First we consider the $\Pin^-(2)$-monopole equations on the canonical $\Spincm$ structure.
Let $\mathfrak{s}_0$ be the canonical $\Spinc$ structure on $(X,\omega)$ and $\hat{\mathfrak{s}}_0$ the  canonical $\Spincm$ structure on $X\to\X$.
Recall $\hat{\omega}$ is a $\ell_{\R}$-valued self-dual $2$-form such that $\omega=\pi^*\hat{\omega}$.
Normalize the metric on $X$ so that $|\hat{\omega}|=\sqrt2$ and pull it back to $\X$ so that $|\omega|=\sqrt2$.
Recall the splitting 
\[
\Lambda^+(X)\otimes_{\R}\C=\C\cdot \omega\oplus K\oplus K^{-1}.
\]
The Clifford multiplication by $\omega$ induces the splitting $W^+_0=\underline{\C}\oplus K^{-1}$.
In fact, the components $\underline{\C}$ and $K^{-1}$ are respectively $(+2)$ and $(-2)$-eigenspaces of the action of $(\omega/i)$ on $W^+_0$.
 
On the $\Spincm$ structure $\hat{\mathfrak{s}}_0$, we have a twisted Clifford multiplication $\rho\colon \Lambda^1(\X)\otimes i\ell_{\R}\to \Hom(\hat{W}^+_0,\hat{W}^-_0)$ \cite{N1}, and this extends to 
\[
\rho\colon: \Lambda^+(\X)\otimes i\ell_{\R} \to \End(\hat{W}^+_0).
\]  
Then $(\hat{\omega}/i)$ induces the splitting $\hat{W}^+_0=\underline{\hat{\C}}\oplus  \hat{K}^{-1}$.
Since the real part of $\underline{\hat{\C}}$ is trivial, there is a constant section $\hat{u}_0$ such that $|\hat{u}_0|=1$. 
Mimicking the argument of Taubes~\cite{Taubes1}, we obtain the following.
\begin{Proposition}
There is a unique  $\OO(2)$-connection $\hat{A}_0$ (up to gauge) on $\hat{P}_K$ whose induced covariant derivative $\nabla_{\hat{A}_0}$  on $\hat{W}^+$ has the property that
\[
\left(1+\frac12\rho(\hat{\omega}/i)\right)\nabla_{\hat{A}_0} \hat{u}_0 =0.
\]
Furthermore, $D_{\hat{A}_0}\hat{u}_0=0$ if and only if $d\hat{\omega}=0$, where $D_{\hat{A}_0}$ is the Dirac operator associated with $\hat{A}_0$.
\end{Proposition}
Let us consider the $\Pin^-(2)$-monopole equations rescaled and perturbed as follows:
\begin{equation}\label{eq:pin2omega}
D_{\hat{A}}\hat{\phi}=0,\quad F^+_{\hat{A}} = rq(\hat{\phi}) - \frac{r}4i\hat{\omega} + F^+_{\hat{A}_0},
\end{equation}
where $\hat{A}$ is an $\OO(2)$-connection on $\hat{P}_K$, $\hat{\phi}\in\Gamma(\hat{W}^+_0)$, $q$ is the quadratic form defined in \cite{N1} and  $r$ is a positive real constant. 
(This is an analogue of Taubes' perturbation \cite{Taubes3}.)
Then we can see that $(\hat{A}_0,\hat{u}_0)$ is a solution to \eqref{eq:pin2omega} for every $r$.

To proceed further, it is convenient to move to the upstairs and consider the $I$-invariant part.
Let $(A_0,u_0)$ be the lift of  $(\hat{A}_0, \hat{u}_0)$.
Then a spinor $\phi\in\Gamma(W^+_0)$ can be written as $\phi=\alpha u_0 + \beta$, where $\alpha$ is a complex-valued function on $X$ and $\beta\in\Gamma(K^{-1})$.
Then a solution to the equation \eqref{eq:pin2omega} corresponds to an $I$-invariant solution to the perturbed equation due to Taubes \cite{Taubes3}:
\begin{equation}\label{eq:swomega}
\begin{aligned}
D_{A}\phi &=0,\\
F_{A}^+-F_{A_0}^+ &=-\frac{ir}8(1-|\alpha|^2+|\beta|^2)  \omega+ \frac{ir}4(\alpha\bar{\beta}+\bar{\alpha}\beta).
\end{aligned}
\end{equation}
\begin{proof}[Proof of \thmref{thm:canonical}]
By construction, $({A}_0,{u}_0)$ is an $I$-invariant solution to \eqref{eq:swomega} for every $r$.
Taubes~\cite{Taubes1, Taubes2, Taubes3}(see also Kotschick \cite{Kot}) proved that there is no solution  to \eqref{eq:swomega} except $({A}_0,{u}_0)$ for large $r$.
It follows from this that $(\hat{A}_0, \hat{u}_0)$ is a unique solution to \eqref{eq:pin2omega} for large $r$.
These imply \thmref{thm:canonical}.
\end{proof}
\begin{proof}[Proof of  \thmref{thm:more}]
Suppose the $\Pin^-(2)$-monopole invariant on a $\Spincm$ structure $\hat{\mathfrak{s}}$ is nonzero.
Then the equations \eqref{eq:pin2omega} on $\hat{\mathfrak{s}}$ has a solution $(\hat{A},\hat{\phi})$ for every $r$.
Then the lift $(A,\phi)$ of $(\hat{A},\hat{\phi})$ is an $I$-invariant solution to \eqref{eq:swomega} for $r$. 
By Kotschick \cite{Kot}({\it Cf.} Taubes~\cite{Taubes2}), the existence of solutions for large $r$ implies that 
\begin{equation}\label{eq:ineq}
\left|c_1(L)\cdot[\omega]\right|\leq c_1(K)\cdot[\omega],
\end{equation}
where $L$ is the determinant line bundle of $(X, \mathfrak{s})$.
Let $\hat{L}$ be the characteristic bundle for $(\X,\hat{\mathfrak{s}})$. 
Then  $L=\pi^*\hat{L}$.
The inequality \eqref{eq:more} follows from \eqref{eq:ineq}

By \cite{Taubes3}, we can find an embedded symplectic curve $C$ in $X$ such that $e=P.D.[C]$ satisfies $e^2 = c_1(K\otimes L)$.
If $X$ contains embedded $2$-spheres with self-intersection number $-1$, then blow them down.
The resulting manifold is a minimal symplectic manifold $X^\prime$ with another embedded symplectic curve $C^\prime$ (see, e.g.,\cite{OO}).
Then the proof of Theorem 0.2(6) of \cite{Taubes3} implies that the virtual dimension $d(\mathfrak{s})$ of the moduli space for $(X, \mathfrak{s})$ is $0$.
Therefore $d(\hat{\mathfrak{s}})=\frac12d(\mathfrak{s})=0$.
\end{proof}
%
%
%
%
\section{Real K\"{a}hler surfaces}\label{sec:realKL}
%
%
The purpose of this section is to prove that the $\Pin^-(2)$-monopole moduli space on a real K\"{a}hler surface can be identified with the $I$-invariant moduli space of holomorphic simple pairs, and the space of $I$-invariant effective divisors.
The moduli space of vortices is also introduced as an intermediate object.
The goal of this section is \corref{cor:isom1} and \corref{cor:isom}.

Let $(X,\omega,\iota)$ be a compact K\"{a}hler surface with anti-holomorphic free involution $\iota$ such that $\iota^*\omega=-\omega$.
Note that the pull-back of a $(p,q)$-form by the anti-linear map $\iota$ is a $(q,p)$-form, and the complex conjugation of a $(q,p)$-form is a $(p,q)$-form.
We define the involution $I$ on the space of $(p,q)$-forms $\Omega^{p,q}(X)$ by
\[
I(\alpha)=\overline{\iota^*\alpha},\quad \alpha\in\Omega^{p,q}(X).
\] 
Note that $K=\Lambda^{2,0}(X)$, $K^{-1}=\Lambda^{0,2}(X)$, $\iota^*K=K^{-1}$.

Suppose that there is a canonical $\Spincm$ structure $\hat{\mathfrak{s}}_0$ on $X\to \hat{X}=X/\iota$.
As explained in \subsecref{subsec:canonical}, every $\Spincm$ structure on $X\to \X$ is obtained from $\hat{\mathfrak{s}}_0$ and an $\R^2$-bundle $\hat{E}$ twisted along $\ell_{\R}$ as $\hat{\mathfrak{s}}_0\hat{\otimes}\hat{E}$.

For a $\Spincm$ structure $\hat{\mathfrak{s}}=\hat{\mathfrak{s}}_0\hat{\otimes}\hat{E}$, there exists a $\Spinc$ structure $\mathfrak{s}=\mathfrak{s}_0\otimes E$ on $X$ which is the canonical reduction of $\pi^*\hat{\mathfrak{s}}$, whose positive spinor bundle is $W^+=E\oplus(E\otimes K^{-1})$ such that $E\cong \pi^*\hat{E}$ as $\R^2$-bundles.
Note that $\iota^*\mathfrak{s}=\bar{\mathfrak{s}}$. 
Then $\iota^* E=\bar{E}$ and  $E$ naturally admits a Hermitian metric $h$ such that $\iota^*h=\bar{h}$.

Let $C$ be a Hermitian connection on $K^{-1}$ induced by the Chern connection on $TX$  associated with the K\"{a}hler structure. 

Recall that the Dirac operator $D$ on the canonical $\Spinc$ structure $\mathfrak{s}_0$ is identified with 
\[
D=\sqrt{2}(\bar{\partial}+\bar{\partial}^*)\colon \Omega^{0,0}(X)\oplus\Omega^{0,2}(X)\to\Omega^{0,1}(X).
\]
Since $\iota$ is anti-holomorphic, the pull-back of $D$ by $\iota$ is 
\[
\iota^*D = \sqrt2(\partial+\partial^*)\colon\Omega^{0,0}(X)\oplus\Omega^{2,0}(X)\to\Omega^{1,0}(X).
\]
Then we see that the Dirac operator $D$ is $I$-equivariant.

Next we consider Dirac operators on a $\Spinc$ structure $\mathfrak{s}=\mathfrak{s}_0\otimes E$. 
For a Hermitian connection $A$ on $\det(W^+)=E^2\otimes K^{-1}$, there is a unique Hermitian connection $B$ on $E$ such that $A=C\otimes B^{\otimes 2}$.
Then the Dirac operator $D_A$ associated with $A$ is identified with 
\[
D_A = \sqrt2(\bar{\partial}_B+\bar{\partial}_B^*)\colon \Omega^{0,0}(E)\oplus\Omega^{0,2}(E)\to\Omega^{0,1}(E).
\]
The pull-back $B^\prime=\iota^*B$ is a Hermitian connection on $\bar{E}=\iota^*E$, and the pull-back $\iota^*D_A$ can be written as 
\[
\iota^*D_A = \sqrt2(\partial_{B^\prime}+\partial_{B^\prime}^*)\colon \Omega^{0,0}(\bar{E})\oplus\Omega^{2,0}(\bar{E})\to\Omega^{1,0}(\bar{E}).
\]

For an $\OO(2)$-connection $\hat{B}$ on $\hat{E}$, we have a Hermitian connection $B$ on $E$ which is the $\U(1)$-reduction of $\pi^* \hat{B}$.
Then $B$ is $I$-invariant, i.e., $B=\overline{\iota^*B}$.
For such a connection $B$, the Dirac operator $D_A=\sqrt2(\bar{\partial}_B+\bar{\partial}_B^*)$ is also $I$-equivariant.

Recall the identifications:
\[
H^2(X;\C)=H^{1,1}\oplus H^{2,0}\oplus H^{0,2},\quad H_+(X;i\R) = i\R\omega\oplus H^{0,2}.
\]
Note that the $I=\overline{\iota^*(\cdot)}$-action preserves $H_+(X;i\R)$ and 
\[
H_+(\X;\ell_{\R})\cong H_+(X;i\R)^I = i\R\omega\oplus (H^{0,2})^I.
\]
In particular, we have the following.
\begin{Proposition}
If $b_+^{\ell}=\rank H_+(\X;\ell_{\R})\geq 2$, then $(H^{2,0})^I\cong(H^{0,2})^I\neq \emptyset$.
\end{Proposition}
%
%
\subsection{Seiberg--Witten equations}
%
%
The Seiberg--Witten equations on K\"{a}hler surfaces can be written as follows(\cite{Teleman,Morgan}):
\begin{equation}\label{eq:SWKL}
\begin{aligned}
\bar{\partial}_B\alpha + \bar{\partial}_B \beta = & 0\\
2F_B^{0,2} + 2\pi i\eta^{0,2}-\frac12\beta\bar{\alpha} =&0\\
2F_B^{2,0} + 2\pi i\eta^{2,0}+\frac12\alpha\bar{\beta} =&0\\
\{\Lambda_g(F_B + \pi i \eta)-\frac{i}2s_g + \frac{i}8(|\beta|^2 -|\alpha|^2)\}\omega = &0
\end{aligned}
\end{equation}
These are the equations for  Hermitian connections $B$ on $E$ and sections $(\alpha,\beta)\in (\Omega^{0,0}\oplus\Omega^{0,2})(E)$.
The perturbation term is given by $\eta\in\Omega^2(X)$,  $\Lambda_g$ denotes the adjoint of the multiplication operator $\omega\wedge\cdot\colon i\Omega^{0,0}\to i\Omega^{1,1}$, and $s_g$ is the scalar curvature. 
(Here we use the fact that $i\Lambda_gF_C=s_g$ for the Chern connection $C$.)
If we take an $I$-invariant $\eta$, then \eqref{eq:SWKL} is $I$-equivariant.

The discussion below is largely indebted to Teleman's excellent exposition \cite{Teleman}.
The general principle is to ``consider in the upstairs and take the $I$-invariant part".
The next two theorems are obtained by restricting everything to the $I$-invariant part in the corresponding theorems of \cite{Teleman}.
\begin{Theorem}[\cite{Teleman}, Th\'{e}or\`{e}me 8.1.7]\label{thm:SWcpl}
Suppose $\eta$ is an $I$-invariant closed $(1,1)$-form, and
\[
\Theta:=\frac12\langle ([\eta]-2c_1(E)+c_1 (K))\cup[\omega],[X]\rangle\neq 0
\]
Then an $I$-invariant triple $(B,\alpha,\beta)$ is a solution to \eqref{eq:SWKL} if and only if:\\
\Romnum{1}. $\Theta>0$ and 
\begin{equation}\label{eq:SWa}
\beta=0,\quad \bar{\partial}_B\alpha=0,\quad F_B^{0,2}=0,\quad i\Lambda_gF_B + \frac18|\alpha|^2=\pi\Lambda_g\eta-\frac{s_g}2
\end{equation}
\Romnum{2}. $\Theta<0$ and 
\begin{equation}\label{eq:SWb}
\alpha=0,\quad \bar{\partial}^*_B\beta=0,\quad F_B^{0,2}=0,\quad i\Lambda_gF_B - \frac18|\beta|^2=\pi\Lambda_g\eta-\frac{s_g}2
\end{equation}
\end{Theorem}
Let $C^{\vee}$ be the Hermitian connection on $K$ induced from the Chern connection $C$. 
For a Hermitian connection on $E$, let $B^\prime$ be the Hermitian connection on $K\otimes \bar{E}$ such that $B\otimes B^\prime=C^{\vee}$.
For $\beta\in\Omega^{0,2}(E)=\Gamma(E\otimes K^{-1})$, let $\varphi=\bar{\beta}\in\Gamma(\bar{E}\otimes K)$.
Then the condition $\bar{\partial}^*_B\beta=0$ is equivalent to $\bar{\partial}_{B^\prime}\varphi=0$ by the Serre duality, and \eqref{eq:SWb} can be rewritten as
\[
\alpha=0,\quad \bar{\partial}_{B^\prime}\varphi=0,\quad F_{B^\prime}^{0,2}=0,\quad i\Lambda_gF_{B^\prime} + \frac18|\varphi|^2=-\pi\Lambda_g\eta+\frac{s_g}2.
\]

If $\eta$ is not a $(1,1)$-form, then we have the following. 
\begin{Theorem}[\cite{Teleman}, Th\'{e}or\`{e}me 9.3.1]\label{thm:SWml}
Suppose an $I$-invariant $2$-form $\eta$ has a form of  $\eta=\eta^{2,0}\oplus\eta^{1,1}\oplus\overline{\eta^{2,0}}$
 where $\eta^{2,0}$ is an $I$-invariant non-zero holomorphic $2$-form.
Then an $I$-invariant triple $(B,\alpha,\beta)$ is a solution to \eqref{eq:SWKL} if and only if:
\begin{equation}\label{eq:SWcpl}
\begin{gathered}
\alpha\bar{\beta} = -8\pi i\eta^{2,0},\quad \bar{\partial}_B\alpha=\bar{\partial}^*_B\beta=0, \quad F_B^{0,2}=0, \\
i\Lambda_gF_B+\frac18(|\beta|^2-|\alpha|^2)=\pi\Lambda_g\eta^{1,1}-\frac{s_g}2
\end{gathered}
\end{equation}
\end{Theorem}
Let $C^{\vee}$, $B^\prime$ and $\varphi$  be as above. 
Then \eqref{eq:SWcpl} can be rewritten as
\begin{equation}\label{eq:Vcpl}
\begin{gathered}
\alpha\varphi = -8\pi i\eta^{2,0},\quad B\otimes B^\prime=C^{\vee},\\
\bar{\partial}_B\alpha=\bar{\partial}_{B^\prime}\varphi=0, \quad F_B^{0,2}=F_{B^\prime}^{0,2}=0, \\
\dfrac{i}2\Lambda_g(F_B-F_{B^\prime})+\frac18(|\varphi|^2-|\alpha|^2)=\pi\Lambda_g\eta
\end{gathered}
\end{equation}
%
%
\subsection{Vortex equations}
%
%
Let $(X,\omega,\iota)$ be a compact K\"{a}hler surface with anti-holomorphic free involution $\iota$.
Suppose we have  a $C^\infty$ Hermitian line bundle $(E,h)$ over $X$ with an isomorphism $\iota^*(E,h)\cong(\bar{E},\bar{h})$. 
This isomorphism defines the bundle map $I=\overline{\iota^*(\cdot)}$ covering $\iota$ which is the composite map of 
\[
\begin{CD}
E@>{\iota^*}>>\iota^*E\cong \bar{E}@>{\bar{(\cdot)}}>>E.
\end{CD}
\]
We suppose $I$ generates an order-2 action (involution) on $E$.
We define the $I$-action on $\Omega^0(E)$ also by $I=\overline{\iota^*(\cdot)}$.
Let $\mathcal{A}(E,h)$ be the space of Hermitian connections on $E$.
Then the involution $I$ naturally induces an involution on $\mathcal{A}(E,h)$, also denoted by $I$.
The gauge transformation group $\G=C^\infty(X;S^1)$ acts on $ \mathcal{A}(E,h)\times\Omega^0(E)$ by
\[
(B,\phi)\cdot f = (B+f^{-1}df,f^{-1}\phi)\text{ for } (B,\phi)\in\mathcal{A}(E,h)\times\Omega^0(E), f\in\G.
\]
A configuration $(B,\phi)$ with $\phi\neq 0$ is called an {\it irreducible}.
The group $\G$ acts on the space of irreducibles freely.
We define the involution $I$ on $\G$ by $I(f)=\overline{\iota^* f}$.
Then the $\G$-action on $\mathcal{A}(E,h)\times\Omega^0(E)$ is $I$-equivariant.

\begin{Definition}
Let $t\colon X\to \R$ be a $C^\infty$-function.
A $t$-vortex is a solution $(B,\phi)\in \mathcal{A}(E,h)\times\Omega^0(E)$ to the system of the equations 
\begin{equation}\label{eq:vortex}
\begin{gathered}
\bar{\partial}_B\phi=0\\
F_B^{0,2} = 0\\
i\Lambda_gF_B + \frac12|\phi|^2 -t =0
\end{gathered}
\end{equation}
If $(B,\phi)$ is a solution to \eqref{eq:vortex}, then 
\begin{equation}\label{eq:nonneg}
\Xi:=\frac{1}{2\pi}\int_Xtdvol_g - \langle c_1(E)\cup[\omega],[X]\rangle = \frac1{2\pi}\int_X(t-i\Lambda_gF_B)dvol_g=\frac1{4\pi}\|\phi\|^2_{L^2}\geq0.
\end{equation}
\end{Definition}
If $t$ is $\iota$-invariant, that is, $\iota^* t =t$, then the system \eqref{eq:vortex} is $I$-equivariant.
Define $I$-invariant moduli spaces as follows:
\begin{align*}
\mathcal{V}_t(E)^I=&\{ \text{ $I$-invariant $t$-vortices }\}/\G^I,\\
\mathcal{V}_t^*(E)^I=&\{ \text{ $I$-invariant irreducible $t$-vortices }\}/\G^I.
\end{align*}
If $\Xi > 0$, then $\mathcal{V}_t^*(E)^I=\mathcal{V}_t(E)^I$.

As usual, we take $L^2_k$-completion of $ \mathcal{C}^*(E) := \mathcal{A}(E,h)\times (\Omega^0(E)\setminus \{0\})$ and $L_{k+1}^2$-completion of $\G$ for sufficiently large $k$.
We use the notation $(\cdot)_k$  for the completed spaces.
For a generic choice of $I$-invariant $t$ with positive $\Xi$, the space $\mathcal{V}_t^*(E)^I$ is a submanifold of the Hilbert manifold 
\[
(\mathcal{B}_k^*)^I := ( \mathcal{C}^*(E)_k)^I/\G_{k+1}^I.
\]
For an orbit $[v]=[(B,\phi)]\in(\mathcal{B}_k^*)^I$, the tangent space of $(\mathcal{B}_k^*)^I$ at $[v]$ is given by
\[
T_{[v]}(\mathcal{B}_k^*)^I =\{(\dot{B},\dot{\phi})\,|\, d^*\dot{B} - i\Ima (\dot{\phi}\bar{\phi})=0\,\}^I.
\]

The following is a direct consequence of \thmref{thm:SWcpl}.
\begin{Corollary}
Suppose $\eta$ is an $I$-invariant closed $(1,1)$-form.
Let $t=\pi\Lambda_g\eta-s_g/2$.
Then we have the following identifications:
\begin{enumerate}
\item $\M(X,\mathfrak{s}_0\otimes E)^I\cong \mathcal{V}_t^*(E)^I$, if $\Theta>0$.
\item $\M(X,\mathfrak{s}_0\otimes E)^I\cong \mathcal{V}_{-t}^*(K\otimes E^{-1})^I$, if $\Theta<0$.
\end{enumerate}
\end{Corollary}
%
%
\subsection{Holomorphic simple pairs}
%
%
Let $(X,\omega,\iota)$ be a compact K\"{a}hler surface with anti-holomorphic involution $\iota$, and $E$ a $C^\infty$ complex line bundle such that $\iota^*E\cong\bar{E}$.
As before, we suppose $I=\overline{\iota^*(\cdot)}$ generates an involution on $E$. 
We define the $I$-action on $\Omega^0(E)$ also by $I=\overline{\iota^*(\cdot)}$.
Let $\mathscr{A}^{0,1}(E)$ be the space of semiconnections on $E$.
Note that a semiconnection $\delta\in\mathscr{A}^{0,1}(E)$ can be written as $\delta=\bar{\partial}_B$ for some complex linear connection $B$ on $E$.
The involution $I$ naturally induces an involution on $\mathscr{A}^{0,1}(E)$, also denoted by $I$.
The complex gauge transformation group $\G^{\C}=C^\infty(X,\C^*)$ acts on $\mathscr{P}(E)=\mathscr{A}^{0,1}(E)\times\Omega^0(E)$ by 
\[
(\delta,\phi)\cdot f = (\delta\cdot f,f^{-1}\phi)\text{ for } (\delta,\phi)\in \mathscr{P}(E),\quad f\in\G^{\C},
\]
where $\delta\cdot f=f^{-1}\circ \delta\circ f=\delta +f^{-1}\bar{\partial} f$.
A pair $(\delta,\phi)$ with nonzero $\phi$ is called {\it simple}. 
Let $\mathscr{P}^s(E)$ be the space of simple pairs.
Then $\G^{\C}$ acts on $\mathscr{P}^s(E)$ freely.
We define the involution $I$ on $\G^{\C}$ by $I(f)=\overline{\iota^* f}$.
Then the $\G^{\C}$-action on $\mathscr{P}(E)$ is $I$-equivariant.

Let $\mathscr{H}(E)$ be the space of holomorphic pairs:
\[
\mathscr{H}(E)=\{(\delta,\phi)\in \mathscr{A}^{0,1}(E)\times\Omega^0(E)\,|\, \delta\circ\delta=0,\delta\phi=0\}.
\]
A pair $(\delta,\phi)\in\mathscr{H}(E)$ with non-zero $\phi$ is called a {\it holomorphic simple pair}.
Let $\mathscr{H}^s(E)$ be the space of holomorphic simple pairs.

We consider the $I$-invariant moduli space of holomorphic simple pairs:
\[
\mathscr{M}^s(E)^I = \mathscr{H}^s(E)^I/(\G^{\C})^I.
\]

The deformation complex for an $I$-invariant holomorphic simple pair $\mathfrak{p}=(\delta,\phi)$ is given by
\[
(\mathscr{C}_{\mathfrak{p}})^I = (\mathscr{C}^0)^I\overset{\mathscr{D}^0_{\mathfrak{p}}}{\to}(\mathscr{C}^1)^I\overset{\mathscr{D}^1_{\mathfrak{p}}}{\to}(\mathscr{C}^2)^I\overset{\mathscr{D}^2_{\mathfrak{p}}}{\to}(\mathscr{C}^3)^I,
\]
where 
\begin{gather}
\mathscr{C}^0 = \Omega^{0,0}(X), \quad \mathscr{C}^i = \Omega^{0,i}(X)\oplus \Omega^{0,i-1}(E)\ (i=1,2),\quad  \mathscr{C}^3 = \Omega^{0,2}(E),\\\label{eq:diff}
\mathscr{D}^i_{\mathfrak{p}}(\alpha,\sigma)=(\bar\partial\alpha, -\delta\sigma-\alpha\phi).
\end{gather}

The moduli space $\mathscr{M}^s(E)^I$ has a Kuranishi model as follows.
\begin{Proposition}[{\it Cf.} \cite{Teleman}, Proposition 8.2.10]
Let $H^i((\mathscr{C}_{\mathfrak{p}})^I)$, $\mathbb{H}^i((\mathscr{C}_{\mathfrak{p}})^I)$ be the cohomology group and harmonic space of the elliptic complex $(\mathscr{C}_{\mathfrak{p}})^I$.
There exists a neighborhood $U_{\mathfrak{p}}$ of $0\in \mathbb{H}^1 ((\mathscr{C}_{\mathfrak{p}})^I)$ and a smooth map 
\[
\mathfrak{t}_{\mathfrak{p}}\colon U_{\mathfrak{p}}\to \mathbb{H}^2((\mathscr{C}_{\mathfrak{p}})^I)
\]
such that a neighborhood of $\mathfrak{p}\in\mathscr{M}^s(E)^I$ is homeomorphic to $\mathfrak{t}_{\mathfrak{p}}^{-1}(0)$.
Furthermore, if $H^2((\mathscr{C}_{\mathfrak{p}})^I)=0$, then $\mathscr{M}^s(E)^I$ is a smooth manifold of dimension $\dim H^1((\mathscr{C}_{\mathfrak{p}})^I)$ near $[\mathfrak{p}]$, and the tangent space of $\mathscr{M}^s(E)^I$ at $[\mathfrak{p}]$ is identified with $H^1((\mathscr{C}_{\mathfrak{p}})^I)$.
\end{Proposition}
The proof is standard.

%
%
\subsection{$I$-invariant divisors}\label{subsec:I-div}
%
%
(A reference of this subsection is \cite{silhol}, \Romnum{1}.4.)
A Weil divisor is a formal linear combination $\sum_{i}n_i D_i$ of irreducible analytic hypersurfaces.
Define the $I$-action on divisors by $I\cdot D=\sum_i n_i\iota(D_i)$.
We call a divisor $D$ {\it $I$-invariant} if $D=I\cdot D$.
We will mainly consider effective divisors, i.e.,  $D=\sum_in_iD_i$ with $n_i\geq 0$.

When $D$ is considered as a Cartier divisor, the $I$-action can be written as follows.
For an open subset $U\subset X$ and a holomorphic function $f\in\mathscr{O}_X(U)$, define $I\cdot f\in\mathscr{O}_X(\iota (U))$  by $(I\cdot f)(x)=\overline{f(\iota x)}$.
Let $\mathscr{S}$ be the set of pairs $(U_\lambda,\lambda)$ where $U_\lambda$ is an open set and $\lambda\in\mathscr{O}_X(U_\lambda)$.
Then define the $I$-action on $\mathscr{S}$ by
\[
I\cdot (U_\lambda,\lambda) = (\iota(U_\lambda),I\cdot\lambda)
\]
An effective Cartier divisor  is given as a subset $\mathscr{F}\subset\mathscr{S}$ whose elements $(U_\lambda,\lambda)$ satisfy the following:
\begin{enumerate}
\item $\lambda$ is not identically zero.
\item $\bigcup_{\lambda\in\mathscr{F}}U_\lambda =X$.
\item For every $\lambda,\mu\in\mathscr{F}$, there exists $g_{\lambda,\mu}\in\mathscr{O}_X^*(U_\lambda\cap U_\mu)$ such that $\lambda = g_{\lambda,\mu}\mu$.
\end{enumerate}
We will take a maximal one of such systems for $\mathscr{F}$.
The effective Weil divisor corresponding to an effective Cartier divisor is obtained by considering $\lambda$ as  local defining equations.
Let 
\[
 I\cdot\mathscr{F} = \{I\cdot(U_\lambda,\lambda)\,|\, (U_\lambda,\lambda)\in\mathscr{F}\}.
\]
Then $I\cdot D$ corresponds to $I\cdot\mathscr{F}$.
Note that $g_{I\lambda,I\mu}=Ig_{\lambda,\mu}$.
If $D$ is $I$-invariant, then we can take $\mathscr{F}$ corresponding to $D$ such that $\mathscr{F}=I\cdot\mathscr{F}$.

The system of cocycles $\{g_{\lambda,\mu}\}$ and local functions $\{\lambda\}$ defines a holomorphic line bundle $\mathscr{L}$ with a holomorphic section $\phi$ as follows.
\begin{gather*}
\mathscr{L}=\left(\bigcup_{\lambda\in\mathscr{F}}\{\lambda\}\times U_\lambda\times \C\right)/\sim,\\
\{\mu\}\times(U_\lambda\cap U_\mu)\times\C\ni(\mu,u,\zeta)\sim(\lambda,u,g_{\lambda,\mu}\zeta)\in \{\lambda\}\times(U_\lambda\cap U_\mu)\times\C,\\
\phi(u)=[(\lambda,u,\lambda(u))]\mod \sim\quad (u\in U_\lambda).
\end{gather*}
Then the corresponding divisor $D$ is $D=Z(\phi)$.
When $(\mathscr{L}_D,\phi_D)$ is associated with $D$ (or $\mathscr{F}$), note that the line bundle with section associated with $I\cdot D$ (or $I\cdot\mathscr{F}$) is 
\[
(\mathscr{L}_{I\cdot D},\phi_{I\cdot D}) = (\overline{\iota^*\mathscr{L}_D},\overline{\iota^*\phi_D}).
\]
If $D$ is $I$-invariant, then an anti-linear involution $I$ on $\mathscr{L}_D$ covering $\iota$ is naturally defined by
\[
I\cdot[(\lambda,u,\zeta)] = [(I\lambda,\iota(u),\zeta)].
\]
%
%
\subsection{$I$-equivariant sheaves}
%
%
Under the $I$-action, the structure sheaf $\mathscr{O}_X$ is an $I$-equivariant sheaf in the sense of \cite{Groth,Stieg}, i.e., the sheaf projection $\mathscr{O}_X\to X$ is $I$-equivariant. 
If $D$ is $I$-invariant, then $\mathscr{O}_X(D)$ and $\mathscr{O}_D(D) = \mathscr{O}_X(D)/\mathscr{O}_X$ are also $I$-equivariant.
For an $I$-equivariant sheaf $\mathscr{E}$, the equivariant sheaf cohomology $H^p(X;I,\mathscr{E})$ is defined:
For an $I$-invariant open set $U\subset X$, let $\Gamma^I(U;\mathcal{E})$ be the module  of $I$-invariant sections.
Take an injective resolution $\mathcal{J}^*(\mathcal{E})$ of $\mathcal{E}$ in the category of $I$-equivariant sheaves. 
Then $H^p(X;I,\mathscr{E})$ is defined by
\[
H^p(X;I,\mathscr{E}) = H^p(\Gamma^I(X; \mathcal{J}^*(\mathcal{E})).
\]
The equivariant direct image $\pi^I\mathcal{E}$ of $\mathcal{E}$ is the sheaf on $\X=X/\iota$ which is generated by the presheaf,
\[
\hat{U}\mapsto \Gamma^I(\pi^{-1}(\hat{U});\mathcal{E}),\quad \hat{U}\subset \X =X/\iota\quad \text{open}.
\] 
In general, $\pi^G$ is a left exact functor for $G$-sheaves. 
However our case is much simple. 
Since $I$ covers the free involution $\iota$ on $X$, $\pi^I$ is an exact functor. 
That is, for an exact sequence of $I$-sheaves  on $X$,
\[
0\to\mathcal{E}\to\mathcal{F}\to\mathcal{H}\to 0,
\]
we have an exact sequence of sheaves on $\X$,
\[
0\to\pi^I\mathcal{E}\to \pi^I\mathcal{F}\to \pi^I\mathcal{H}\to 0.
\]
In particular,
\[
\pi^I\mathcal{F}/\pi^I\mathcal{E} = \pi^I(\mathcal{F}/\mathcal{E}).
\]
The fact that $I$ covers the free involution $\iota$ on $X$ also implies that 
\[
H^i(X;I,\mathscr{E}) = H^i(X/\iota;\pi^I(\mathscr{E})).
\]
(\cite{Groth}, p.204, Corollaire; \cite{Stieg}, Corollary 5.6.)

There is an $I$-equivariant exact sequence.
\[
\begin{CD}
0@>>>\tilde{\Z}@>{i}>>\mathscr{O}_X@>{\exp2\pi}>>\mathscr{O}^*_X@>>> 0,
\end{CD}
\] 
where $\tilde\Z$ is the constant sheaf on which $I$ acts via multiplication of $-1$.
This induces the sequence
\[
0\to H^1(X;I,\tilde\Z)\to H^1(X;I,\mathscr{O}_X)\to H^1(X;I,\mathscr{O}^*_X) \overset{\tilde{c}_1}{\to} H^2(X;I,\tilde\Z)\to\cdots.
\]
Note that $H^i(X;I,\tilde\Z)\cong H^i(\X;\ell)$. 
Let $\mathrm{NS}^I(X)=\Ima \tilde{c}_1$.
For $e\in \mathrm{NS}^I(X)$, let $\mathscr{D}(e)$ be the set of effective divisors representing $e$.
\begin{Proposition}[\cite{Teleman}, Proposition 8.2.13]
Let $e=\tilde{c}_1(E)$. 
\begin{enumerate}
\item The map $(\delta,\phi)\mapsto Z(\phi)$ induces a bijection $\mathscr{M}^s(E)^I\overset{\cong}{\to}\mathscr{D}(e)^I$.
\item 
\[
H^{0}((\mathscr{C}_{\mathfrak{p}})^I)=0,\quad H^{i}((\mathscr{C}_{\mathfrak{p}})^I)\cong H^{i-1}(X;I,\mathscr{O}_D(D))=H^{i-1}(\X;\pi^I\mathscr{O}_D(D)),
\]
for each positive integer $i$.
\end{enumerate}
\end{Proposition}
\begin{proof}
With \subsecref{subsec:I-div} understood, (1) is easy.
The proof of (2) is parallel to that of \cite[Proposition 8.2.13]{Teleman}.
Let $\mathcal{A}^{p,q}(X)$ and $\mathcal{A}^{p,q}(E)$ be the sheaves of $C^\infty$-sections of $\Lambda^{p,q}$ and $\Lambda^{p,q}(E)$.
Let 
\[
\mathcal{C}^0:=\mathcal{A}^0(X),\quad \mathcal{C}^i:=\mathcal{A}^{0,i}(X)\oplus \mathcal{A}^{0,i-1}(E), (i=1,2),\quad \mathcal{C}^3:=\mathcal{A}^{0,2}(E).
\]
Then the $I$-action makes $\mathcal{C}^i$ $I$-equivariant sheaves.
The operators $\mathscr{D}_{\mathfrak{p}}^i$ in \eqref{eq:diff} define a sequence of $I$-equivariant sheaves:
\[
0\to\mathcal{C}^0\overset{\delta^0_{\mathfrak{p}}}{\to}\mathcal{C}^1\overset{\delta^1_{\mathfrak{p}}}{\to}\mathcal{C}^2\overset{\delta^2_{\mathfrak{p}}}{\to}\mathcal{C}^3\to 0.
\]
This induces the sequence of sheaves over $\X=X/\iota$:
\begin{equation}\label{eq:seq}
0\to\pi^I\mathcal{C}^0\overset{\hat{\delta}^0_{\mathfrak{p}}}{\to}\pi^I\mathcal{C}^1\overset{\hat{\delta}^1_{\mathfrak{p}}}{\to}\pi^I\mathcal{C}^2\overset{\hat{\delta}^2_{\mathfrak{p}}}{\to}\pi^I\mathcal{C}^3\to 0.
\end{equation}
Since $\iota$ is free, it can be seen from the $\bar{\partial}$-Poincar\'{e} lemma that the sequence \eqref{eq:seq} is exact unless $i\neq 1$.
Furthermore, the following map is an isomorphism:
\[
\pi^I\mathscr{O}_X(\mathscr{E}_\delta)/\phi\pi^I\mathscr{O}_X\to \ker \hat{\delta}^1_{\mathfrak{p}} / \im \hat{\delta}^0_{\mathfrak{p}}, \quad [\lambda]\mapsto (0,\lambda).
\]
For $i>0$,
\[
H^i(\pi^I\mathcal{C}^1/\im \hat{\delta}^0_{\mathfrak{p}}) =0. 
\]
Then we obtain a resolution of $\pi^I\mathscr{O}_X(\mathscr{E}_\delta)/\phi\pi^I\mathscr{O}_X = \pi^I\mathscr{O}_D(D)$ as follows:
\[
0\to \pi^I\mathscr{O}_X(\mathscr{E}_\delta)/\phi\pi^I\mathscr{O}_X = \ker \hat{\delta}^1_{\mathfrak{p}} / \im \hat{\delta}^0_{\mathfrak{p}}\to \pi^I\mathcal{C}^1/\im \hat{\delta}^0_{\mathfrak{p}}\overset{\hat{\delta}^1_{\mathfrak{p}}}{\to}\pi^I\mathcal{C}^2\overset{\hat{\delta}^2_{\mathfrak{p}}}{\to}\pi^I\mathcal{C}^3\to 0.
\]
Since $H^1(\im\hat{\delta}^0_{\mathfrak{p}}) \cong H^1(\pi^I\mathcal{C}^0)$, we have
\[
H^0(\pi^I\mathcal{C}^1)/\mathscr{D}^0_{\mathfrak{p}}(H^0(\pi^I\mathcal{C}^0)) \cong H^0(\pi^I\mathcal{C}^1/\im\hat\delta^0_{\mathfrak{p}}).
\]
Then we obtain
\begin{align*}
H^0(\pi^I\mathscr{O}_D(D))\cong &\ker\left(H^0(\pi^I\mathcal{C}^1/\im\hat\delta^0_{\mathfrak{p}})\overset{\mathscr{D}_{\mathfrak{p}}^1}{\to}H^0(\pi^I\mathcal{C}^2) = (\mathscr{C}^2)^I\right) = H^1\left((\mathscr{C}_{\mathfrak{p}})^I\right),\\
H^i(\pi^I\mathscr{O}_D(D)) \cong & \ker \mathscr{D}_\mathfrak{p}^{i+1}/\im\mathscr{D}_\mathfrak{p}^{i}=H^{i+1}\left((\mathscr{C}_{\mathfrak{p}})^I\right).
\end{align*}
\end{proof}
\begin{Corollary}
For $D\in\mathscr{D}(e)^I$, if   $H^1(X;I,\mathscr{O}_D(D)) =H^1(\X;\pi^I\mathscr{O}_D(D))=0$, then $\mathscr{D}(e)^I$  is smooth at $D$, and the tangent space of $\mathscr{D}(e)^I$ at $D$ is identified with $H^0(X;I,\mathscr{O}_D(D)) =H^0(\X;\pi^I\mathscr{O}_D(D))$.
\end{Corollary}
We call the following the Zariski tangent space of $\mathscr{D}(e)^I$ at $D$:
\[
T_D(\mathscr{D}(e)^I):= H^0(X;I,\mathscr{O}_D(D)) =H^0(\X;\pi^I\mathscr{O}_D(D)).
\]
%
%
\subsection{Correspondence}\label{subsec:corres}
%
%
Define the map $\tilde{\mathcal{J}}\colon \mathcal{C}^*(E)_k\to \mathscr{P}^s(E)_k$ by 
\[
\tilde{\mathcal{J}}(B,\phi) =(\bar{\partial}_B,\phi). 
\]
Then the restriction of $\tilde{\mathcal{J}}$ to the $I$-invariant part $\mathcal{C}^*(E)^I_k$ induces a submersion
\[
\mathcal{J}^\prime\colon(\mathcal{B}^*_k)^I=\mathcal{C}^*(E)^I_k/\G^I_{k+1}\to(\mathscr{B}_k^s)^I:=\mathscr{P}^s(E)_k^I/(\G^{\C}_{k+1})^I.
\]

The goal of this subsection is the next proposition.
\begin{Theorem}\label{thm:corr}
If $\Xi>0$, then the map $\mathcal{J}^\prime$ induces a homeomorphism
\[
\mathcal{J}\colon\mathcal{V}^*_t(E)^I\overset{\cong}{\to}\mathscr{M}^s(E)^I.
\]
\end{Theorem}
For the proof, we need some preparation.
Define $\tilde{\mu}_t\colon \mathcal{C}^*_k\to \Omega^0(E)_{k-1}$ by the left hand side of the third equation of \eqref{eq:vortex} as 
\[
\tilde{\mu}_t(B,\phi)=  i\Lambda_gF_B + \frac12|\phi|^2 -t.
\]
The restriction 
of $\tilde{\mu}_t$ to the $I$-invariant part $(\mathcal{C}^*_k)^I$
induces the map $\mu_t^I\colon(\mathcal{B}_k^*)^I\to\Omega^0(E)_{k-1}^I$.
Let $Z(\mu_t^I)=(\mu^{I}_t)^{-1}(0)$.
For $v=(B,\phi)\in(\mathcal{C}^*_k)^I$, by using the K\"{a}hler identities $\partial^*=i[\Lambda,\bar{\partial}]$, $\bar{\partial}^*=-i[\Lambda,\partial]$, 
we have
\begin{align*}
T_{[v]}Z(\mu_t^I) =& T_{[v]}(\mathcal{B}^*_k)^I\cap \ker d\tilde{\mu}_t\\
=& \{ (\dot{B},\dot{\phi})\,|\, d^*\dot{B} - i\Ima (\dot{\phi}\bar{\phi})=0, -i\Lambda_gd\dot{B}-\mathrm{Re}(\dot{\phi}\bar{\phi})=0 \}^I\\
=&\{(\dot{B},\dot{\phi})\,|\,2\bar{\partial}^*\dot{B}^{0,1}-\dot{\phi}\bar{\phi} = 0\}^I
\end{align*}
Note that the last space can be identified with the $L^2$-orthogonal complement of the tangent space of the orbit $\mathfrak{p}\cdot(\G^{\C}_{k+1})^I$ in $T_{\mathfrak{p}}(\mathscr{P}_k^s)^I$ where $\mathfrak{p}=\tilde{\mathcal{J}}(v)\in(\mathscr{P}^s)^I$.
From this, we obtain the following: 
\begin{Proposition}\label{prop:lochomeo}
The map $\mathcal{J}^\prime|_{Z(\mu_t^I)}\colon Z(\mu_t^I)\to(\mathscr{B}_k^s)^I$  is a local homeomorphism.
\end{Proposition}
\begin{proof}[Proof of \thmref{thm:corr}]
({\it Cf.}  \cite{Teleman}, Proposition 8.2.20.)
By \propref{prop:lochomeo}, it suffices to see that $\mathcal{J}$ is bijective.
First, we prove $\mathcal{J}$ is surjective.
Suppose $(\delta,\phi)\in\mathscr{M}^s(E)^I$.
We have an Hermitian connection $B=A_{h,\delta}$ associated with the holomorphic structure $\delta$.
We want to find an $I$-invariant function $\psi\in C^\infty(X;\R)^I$ such that $(A_{h,\delta\cdot f},\phi)$ is a solution to \eqref{eq:vortex} for $f=e^{-\psi}$.
Note that 
\[
A_{h,\delta\cdot f}  = A_{h,\delta} -\bar{\partial}\psi +\partial\psi, 
\]
and $(A_{h,\delta\cdot f},\phi)$ is a $t$-vortex if and only if
\begin{equation}\label{eq:KW0}
i\Lambda_g\bar\partial\partial\psi + \frac12e^{2\psi}|\phi|^2 = t- i\Lambda_gF_B.
\end{equation}
Since $\phi$ and $\theta:=t- i\Lambda_gF_B$ are $I$-invariant section and function, they descend on $\X$, i.e., we find $\hat{\phi}$ and $\hat{\theta}$ such that $\phi=\pi^*\hat{\phi}$ and $\theta=\pi^*\hat{\theta}$.
Consider the following equation for $\hat{\psi}\in C^\infty(\X;\R)$:
\[
\Delta_{\hat{g}}\hat{\psi} + \frac12e^{2\hat{\psi}}|\hat{\phi}|^2 = \hat{\theta}.
\]
This is  a Kazdan-Warner type equation \cite{KW}, and has a unique solution $\hat{\psi}$ since $\int_{\hat{X}}\hat{\theta}dvol_{\hat{g}} = \frac12\int_X\theta dvol_g =\frac12\Theta>0$.
Then $\psi=\pi^*\hat{\psi}$ is an $I$-invariant solution to \eqref{eq:KW0}.

We prove $\mathcal{J}$ is injective.
Suppose $(B_1,\phi_1)$, $(B_2,\phi_2)$ are $I$-invariant solutions to \eqref{eq:vortex} such that $(\partial_{B_1},\phi_1)=(\partial_{B_2},\phi_2)\cdot f$ for some $f\in(\G^{\C})^I$.
By replacing $(B_2,\phi_2)$ with $\G^I$-equivalent one, if necessary, we may assume $f= e^{-\psi}$ for some $I$-invariant function $\psi$.
Since $(B_2,\phi_2)$ is an $I$-invariant solution, $\psi$ satisfies \eqref{eq:KW0}.
Since $(B_1,\phi_1)$ is an $I$-invariant solution, $\psi=0$ is a solution to \eqref{eq:KW0}.
Moving to the downstairs, we see that the uniqueness of the solution to Kazdan-Warner's equation implies that $\psi=0$.
\end{proof}
Recall that $\hat{\mathfrak{s}}_0\hat\otimes\hat{E}$ is a $\Spincm$ structure  on $X\to\X$ and $\mathfrak{s}_0 \otimes E$ is the canonical reduction of $\pi^*(\hat{\mathfrak{s}}_0\hat\otimes\hat{E})$.
We choose an $I$-invariant closed $(1,1)$-form $\eta$ for the perturbation term of the $I$-invariant Seiberg--Witten equation \eqref{eq:SWKL}. 
\begin{Corollary}\label{cor:isom1}
Let $t=\pi\Lambda_g\eta-s_g/2$, $e=\tilde{c}_1(\hat{E})$ and $k=\tilde{c}_1(\hat{K})$.
\begin{enumerate}
\item If $\Theta>0$, then 
\[
\M(\X,\hat{\mathfrak{s}}_0\hat\otimes\hat{E})\cong \M(X,\mathfrak{s}_0\otimes E)^I\cong \mathcal{V}_t^*(E)^I\cong \mathscr{M}^s(E)^I\cong\mathscr{D}(e)^I.
\]
\item If $\Theta<0$, then 
\[
\M(\X,\hat{\mathfrak{s}}_0\hat\otimes\hat{E})\cong \M(X,\mathfrak{s}_0\otimes E)^I\cong \mathcal{V}_{-t}^*(E^{-1}\otimes K)^I\cong \mathscr{M}^s(E^{-1}\otimes K)^I\cong\mathscr{D}(k-e)^I.
\]
\end{enumerate}
\end{Corollary}
%
\subsection{Witten's perturbation}
%
%
In the previous subsection, we consider the perturbation by an $I$-invariant $(1,1)$-form $\eta$, and the $\Pin^-(2)$-monopole moduli space is identified with the $I$-invariant moduli space of vortices and holomorphic simple pairs.
In this subsection, we consider the perturbation as in \thmref{thm:SWml}.

Let $(X,\omega,\iota)$ be a compact K\"{a}hler surface with anti-holomorphic involution $\iota$, and $E$ and $E^\prime$ two $C^\infty$ complex line bundles such that $\iota^*E\cong\bar{E}$ and $\iota^*E^\prime\cong\bar{E}^\prime$.
We suppose $I=\overline{\iota^*(\cdot)}$ defines involutions on $E$ and $E^\prime$.
Consider $\mathscr{P}^s(E)\times\mathscr{P}^s(E^\prime)$. 
Let $\G^{\C}=C^\infty(X,\C^*)$ act on $\mathscr{P}^s(E)\times\mathscr{P}^s(E^\prime)$ by  
\[
(\mathfrak{p},\mathfrak{p}^\prime)\cdot f :=(\mathfrak{p}\cdot f,\mathfrak{p}^\prime\cdot f^{-1})  
\text{ for }  (\mathfrak{p},\mathfrak{p}^\prime)\in \mathscr{P}^s(E)\times\mathscr{P}^s(E^\prime),\ f\in\G^{\C}.
\]
Fix a holomorphic structure $\mathscr{N}$ on $N=E\otimes E^\prime$, and let $\delta_{\mathscr{N}}$ be the corresponding integrable semiconnection.
(Later we assume $\mathscr{N}$, and therefore $\delta_{\mathscr{N}}$, are  $I$-invariant.)
Let 
\[
\mathscr{H}^s(E)\times_{\mathscr{N}}\mathscr{H}^s(E^\prime) :=\{((\delta,\phi),(\delta^\prime,\phi^\prime)) \in \mathscr{H}^s(E)\times_{\mathscr{N}}\mathscr{H}^s(E^\prime)\,|\,\delta\otimes\delta^\prime = \delta_{\mathscr{N}}\}.
\]
A natural map $\mathscr{T}\colon \mathscr{P}^s(E)\times\mathscr{P}^s(E^\prime)\to \mathscr{P}^s(N)$ given by $((\delta,\phi),(\delta^\prime,\phi^\prime))\mapsto(\delta\otimes\delta^\prime,\phi\otimes\phi^\prime)$ is $\G^{\C}$-invariant.
We have a $\G^{\C}$-equivariant commutative diagram
\[
\begin{CD}
\mathscr{H}^s(E)\times\mathscr{H}^s(E^\prime) @>{\mathscr{T}}>>\mathscr{H}^s(N)\\
@AAA @AAA\\
\mathscr{H}^s(E)\times_{\mathscr{N}}\mathscr{H}^s(E^\prime) @>{\mathscr{T}_{\mathscr{N}}}>>\{\delta_{\mathscr{N}}\} \times (H^0(\mathscr{N})\setminus\{0\})
\end{CD}
\]
Now suppose $\mathscr{N}$ is $I$-invariant and $(H^0(\mathscr{N})\setminus\{0\})^I\neq\emptyset$. 
Choose an $I$-invariant holomorphic section $\xi\in (H^0(\mathscr{N})\setminus\{0\})^I$.
Let
\[
\mathscr{M}^s(E,E^\prime,\mathscr{N},\xi)^I = \left(\mathscr{T}_{\mathscr{N}}^{-1}(\xi)\right)^I/\G_{\C}^I.
\]

For $e=c_1(E)$ and $e^\prime=c_1(E^\prime)$, consider the map defined by sum of divisors
\[
\theta\colon\mathscr{D}(e)\times\mathscr{D}(e^\prime)\to \mathscr{D}(e+e^\prime).
\]
For $\Delta\in\mathscr{D}(e+e^\prime)$, let
\[
\mathscr{D}_b(\Delta) := \theta^{-1}(\Delta).
\]
Then, for $\Delta=Z(\xi)$, we have a natural identification
\[
\mathscr{M}^s(E,E^\prime,\mathscr{N},\xi)^I \cong \mathscr{D}_b(\Delta)^I.
\]
The Zariski tangent space $T_{(D,D^\prime)}\left(\mathscr{D}_b(\Delta)^I\right)$ of $\mathscr{D}_b(\Delta)^I$ at $(D,D^\prime)$ is given by
\[
T_{(D,D^\prime)}\left(\mathscr{D}_b(\Delta)^I\right) = \ker\left(\theta_*\colon T_D\left(\mathscr{D}(e)^I\right) \oplus T_{D^\prime}\left(\mathscr{D}(e^\prime)^I\right) \to T_{\Delta}\left(\mathscr{D}(e+e^\prime)^I\right)\right).
\]
For $I$-invariant $D$, $D^\prime$, $\Delta=D+D^\prime$, the inclusions $\emptyset\subset D\subset\Delta$, $\emptyset\subset D^\prime\subset\Delta$ induce the inclusions
\[
\pi^I\mathscr{O}_X\subset\pi^I\mathscr{O}_X(D)\subset\pi^I\mathscr{O}_X(\Delta),\quad \pi^I\mathscr{O}_X\subset\pi^I\mathscr{O}_X(D^\prime)\subset\pi^I\mathscr{O}_X(\Delta),
\]
\begin{align*}
\pi^I\mathscr{O}_D(D)=\pi^I\left(\mathscr{O}_X(D)/\mathscr{O}_X\right)&\hookrightarrow\pi^I\left(\mathscr{O}_X(\Delta)/\mathscr{O}_X\right)=\pi^I\mathscr{O}_\Delta(\Delta),\\
\pi^I\mathscr{O}_{D^\prime}(D^\prime)=\pi^I\left(\mathscr{O}_X(D^\prime)/\mathscr{O}_X\right)&\hookrightarrow\pi^I\left(\mathscr{O}_X(\Delta)/\mathscr{O}_X\right)=\pi^I\mathscr{O}_\Delta(\Delta),
\end{align*}
and therefore the injective maps
\begin{align*}
T_D\left(\mathscr{D}(e)^I\right) = H^0(\pi^I\mathscr{O}_D(D))\overset{i}{\hookrightarrow}& H^0(\pi^I\mathscr{O}_{\Delta}(\Delta))=T_{\Delta}\left(\mathscr{D}(e+e^\prime)^I\right),\\ 
T_D\left(\mathscr{D}(e^\prime)^I\right) = H^0(\pi^I\mathscr{O}_{D^\prime}(D^\prime))\overset{i^\prime}{\hookrightarrow}& H^0(\pi^I\mathscr{O}_{\Delta}(\Delta))=T_{\Delta}\left(\mathscr{D}(e+e^\prime)^I\right).
\end{align*}
Then it can be seen that
\[
\theta_*(a,a^\prime)=i(a)+i^\prime(a^\prime).
\]
\begin{Proposition}\label{prop:maxdiv}
Let $D_0$ be the maximal effective divisor such that $D_0\leq D$ and $D_0\leq D^\prime$. 
(Note that $D_0$ is also $I$-invariant.)
Then there exists an isomorphism
$T_{(D,D^\prime)}\left(\mathscr{D}_b(\Delta)^I\right) \cong H^0(\pi^I\mathscr{O}_{D_0}(D_0))$.
\end{Proposition}
\begin{proof}
This is proved by considering the $I$-invariant part or applying $\pi^I$ to everything in the proof of  \cite[Lemma 9.3.3]{Teleman}.
The commutative diagram
\[
\xymatrix{&& \pi^I\Os(D)\ar[rd]\\
\pi^I\Os\ar[r]  &\pi^I\Os(D_0)\ar[ru]\ar[rd] & & \pi^I\Os(\Delta)\\
&&\pi^I\Os(D^\prime)\ar[ru]
}
\]
induces a commutative diagram
\[
\xymatrix{& H^0(\pi^I\Os_D(D))\ar[rd]^{i}\\
H^0(\pi^I\Os_{D_0}(D_0))\ar[ru]^u\ar[rd]_{u^\prime}\ar[rr]^v & & H^0(\pi^I\Os_{\Delta}(\Delta))\\
&H^0(\pi^I\Os_{D^\prime}(D^\prime))\ar[ru]_{i^\prime}
}
\]
where all of maps are linear monomorphisms.
The image of the monomorphism 
\[
(-u)\oplus u^\prime\colon H^0(\pi^I\Os_{D_0}(D_0)) \to H^0(\pi^I\Os_D(D))\oplus H^0(\pi^I\Os_{D^\prime}(D^\prime))
\]
is contained in $\ker(\theta_{*})$.
Therefore 
\[
H^0(\pi^I\Os_{D_0}(D_0))\subset \ker(\theta_{*})=T_{(D,D^\prime)}\left(\mathscr{D}_b(\Delta)^I\right).
\]

Conversely, let $(a,a^\prime)$ be an element of $\ker(\theta_{*})$. Then $i(a)+i^\prime(a^\prime)=0$.
We have the exact sequences
\begin{gather*}
0\to\pi^I\Os_{D^\prime}(-D)\to\pi^I\Os_{\Delta}\overset{\rho}{\to}\pi^I\Os_D\to 0,\\
0\to\pi^I\Os_{D}(-D^\prime)\to\pi^I\Os_{\Delta}\overset{\rho^\prime}{\to}\pi^I\Os_{D^\prime}\to 0,\\
0\to\pi^I\Os_{D^\prime}(D^\prime)\overset{i^\prime}{\to}\pi^I\Os_{\Delta}(\Delta)\overset{r}{\to}\pi^I\Os_D(\Delta)\to 0,\\
0\to\pi^I\Os_{D}(D)\overset{i}{\to}\pi^I\Os_{\Delta}(\Delta)\overset{r^\prime}{\to}\pi^I\Os_{D^\prime}(\Delta)\to 0.
\end{gather*}
Then $\rho$, $\rho^\prime$, $r$, $r^\prime$ are restriction maps, since the corresponding maps in the exact sequences without $\pi^I$ called the {\it decomposition sequences} \cite[p.62]{BHPV} are restriction maps.
Since $r^\prime\circ i=r\circ i^\prime=0$ and $i(a)=-i^\prime(a^\prime)$, we have
\[
r^\prime(i(a))=0,\quad r(i(a))=-r(i^\prime(a^\prime))=0.
\]
Hence the restrictions of $i(a)\in H^0(\pi^I\Os_{\Delta}(\Delta))$ to $D$ and $D^\prime$ are $0$.
Let $\tilde{D}\leq \Delta$ be the smallest effective divisor such that $D\leq \tilde{D}$, $D^\prime\leq\tilde{D}$.
Then the restriction of $i(a)$ to $\tilde{D}$ is also $0$.
By using the decomposition $\Delta = D_0+\tilde{D}$, we obtain the exact sequence
\[
0\to\pi^I\Os_{D_0}(D_0)\overset{v}{\to}\pi^I\Os_{\Delta}(\Delta)\overset{\tilde{r}}{\to}\pi^I\Os_{\tilde{D}}(\Delta)\to 0.
\]
Since $\tilde{r}(i(a))=0$, we have an element $b\in H^0(\pi^I\Os_{D_0}(D_0))$ such that $v(b)=i(a)$.
Since $v=i\circ u$ and $i$ is injective, $a=u(b)$.
Then $i^\prime(a^\prime + u^\prime(b))= -i(a)+(i^\prime\circ u^\prime)(b) = -i(a)+v(b)=0$, and hence $a^\prime=-u^\prime(b)$.  
\end{proof}

Fix Hermitian metrics $h$, $h^\prime$ on $E$ and $E^\prime$, a function $t\in C^\infty(X,\R)$, an integrable connection $\Sigma$ on $N$ and a nonzero $\bar{\partial}_{\Sigma}$-holomorphic section $\xi\in \Omega^0(N)\setminus\{0\}$.
(Later we assume that all of them are $I$-invariant.)
Let $\mathscr{N}$ be the holomorphic structure on $N$ induced from $\Sigma$.
Let $\G=C^\infty(X,S^1)$ act on $(\mathcal{A}(E,h)\times \Omega^0(E))\times(\mathcal{A}(E^\prime,h^\prime)\times \Omega^0(E^\prime))$ by
\[
((B,\phi),(B^\prime,\phi^\prime)) \cdot f = ((B,\phi)\cdot f,(B^\prime,\phi^\prime)\cdot f^{-1}) .
\]
Consider the following system of equations:
\begin{equation}\label{eq:WV}
\left\{
\begin{gathered}
\bar{\partial}_B\phi = \bar{\partial}_{B^\prime}\phi^\prime = 0\\
F^{0,2}_B = F^{0,2}_{B^\prime} = 0\\
i\Lambda_g (F_B - F_{B^\prime}) +\frac12(|\phi|^2 -|\phi^\prime|^2)=t\\
B\otimes B^\prime =\Sigma\\
\phi\otimes\phi^\prime = \xi
\end{gathered}\right.
\end{equation}
Suppose that all of $h$, $h^\prime$, $t$, $\Sigma$ and $\eta$ are $I$-invariant.
Let
\[
\mathcal{V}_t(E,E^\prime,\Sigma,\xi)^I=\{\text{ $I$-invariant solutions to \eqref{eq:WV} }\}/\G^I.
\]
\begin{Theorem}
The correspondence
\[
((B,\phi),(B^\prime,\phi^\prime)) \mapsto ((\bar{\partial}_B,\phi), (\bar{\partial}_{B^\prime},\phi^\prime)
\]
induces a homeomorphism
\[
\mathcal{V}_t(E,E^\prime,\Sigma,\xi)^I \overset{\cong}{\to} \mathscr{M}^s(E,E^\prime,\mathscr{N},\xi)^I
\]
\end{Theorem}
\begin{proof}
The proof is similar to that of \thmref{thm:corr}.
In this case, we need to find an $I$-invariant function $\psi\colon X\to \R$ so that 
\[
i\Lambda_g\bar\partial\partial\psi + \frac12e^{2\psi}|\phi|^2 - \frac12e^{-2\psi}|\phi^\prime|^2 = t- i\Lambda_g(F_B-F_{B^\prime}).
\]
As before, this equation descends to $\X$, and it has a unique smooth solution.
(See \cite{Biquard} or \cite[\S3.2]{Nicolaescu}.) 
The rest of the proof is similar.
\end{proof}

\begin{Corollary}\label{cor:isom}
For $\eta$ as in \thmref{thm:SWml}, let $t=\pi\Lambda_g\eta^{1,1}-s_g/2$.
\begin{gather*}
\M(\X,\hat{\mathfrak{s}}_0\hat\otimes\hat{E})\cong \M(X;\mathfrak{s}_0\otimes E)^I\\\cong \mathcal{V}_t(E,E^{-1}\otimes K,C^{\vee},\eta^{2,0})^I \cong \mathscr{M}^s(E,E^{-1}\otimes K,\mathscr{K},\eta^{2,0})^I\cong \mathscr{D}_b(\Delta)^I
\end{gather*}
\end{Corollary}
%
%
\section{Computation and Examples}\label{sec:examples}
%
%
The purpose of this section is to compute $\Pin^-(2)$-monopole invariants of several concrete examples.
\subsection{Surfaces of general type}\label{subsec:general}
%
%
In this subsection, we prove \thmref{thm:general} on the surfaces of general type and give a series of examples of this type.
\begin{proof}[Proof of \thmref{thm:general}]
({\it Cf.} \cite{Morgan}, Theorem 7.4.1.)
The results on the canonical and anti-canonical $\Spincm$ structures follow from \thmref{thm:canonical} and \corref{cor:anticanonical}.

Since $X$ is minimal and of general type, $K_X^2>0$ and $K_X$ is numerically effective.
The latter condition implies $K_X\cdot\omega\geq 0$.
But if $K_X\cdot\omega=0$, then the Hodge index theorem implies $K_X^2\leq 0$.
Therefore $K_X\cdot\omega>0$.

Suppose that $\SW^{\Pin}_X(\hat{\mathfrak{s}})$ is nonzero for a $\Spincm$ structure $\hat{\mathfrak{s}}$.
Let $\mathfrak{s}$ be the canonical reduction of $\pi^*\hat{\mathfrak{s}}$.
Let $L$ be the determinant line bundle of $\mathfrak{s}$.
Then there exists a complex line bundle $E$ such that $\mathfrak{s}=\mathfrak{s}_0\otimes E$.
Note that $L=2E-K_X$.

Since $\SW^{\Pin}_X(\hat{\mathfrak{s}})\neq 0$, $d(\mathfrak{s})=2d(\hat{\mathfrak{s}})\geq 0$, and therefore $L^2\geq K_X^2>0$.
Then $c_1(L)^+$ is not a torsion class, and this implies that there is no reducible solution and $L\cdot\omega\neq 0$.

Suppose $L\cdot\omega>0$.
Since $\SW^{\Pin}_X(\hat{\mathfrak{s}})\neq 0$, there is an $I$-invariant holomorphic structure on $E$ and an $I$-invariant non-zero holomorphic section.
Hence $K_X\cdot E\geq 0$ because $K_X$ is numerically effective.
Since $E$ can be written as $E=(K_X+L)/2$, $K_X\cdot E\geq 0$ implies $K_X^2\geq -K_X\cdot L$.
Since $K_X\cdot\omega>0$ and $L\cdot\omega>0$, there is $t\geq 0$ such that 
\[
\omega\cdot(K_X+tL)=0.
\]
By the Hodge index theorem, we have
\[
0\geq (K_X+tL)^2 = K_X^2 +2tK_X\cdot L+t^2L^2 =:f(t).
\]
The quadratic function $f(t)$  attains its minimum at $t=-(K_X\cdot L)/L^2$ and the minimum is 
\[
K_X^2 -\frac{(K_X\cdot L)^2}{L^2}.
\]
Since $L^2\geq K_X^2\geq -K_X\cdot L$, this quantity is non-negative, and therefore equal to $0$.
Then we have $L^2= K_X^2= -K_X\cdot L$, and we see that $f(t)\leq 0$ only when $t=1$. 
Hence $(K_X+L)^2=0$ and $(K_X+L)\cdot\omega=0$.
By the Hodge index theorem, we have $K_X+L$ is a torsion class, and therefore $E$ is also a torsion class.
Since $E$ has an $I$-invariant non-zero holomorphic section, $E$ is an $I$-equivariant trivial bundle.
This means $\hat{\mathfrak{s}}$ is the canonical $\Spincm$ structure.

On the other hand, in the case when $L\cdot\omega>0$, $\SW^{\Pin}_X(\hat{\mathfrak{s}})\neq 0$ implies the existence of an $I$-invariant holomorphic structure on $E-K_X$, and an $I$-invariant holomorphic section on it.
Arguing similarly, we can prove that $K_X-L$ is a torsion class, and $\hat{\mathfrak{s}}$ is the anti-canonical $\Spincm$ structure.
\end{proof}

For a positive integer $k$, let $M_{4k}$ be the hypersurface in $\CP^3$ defined by real polynomials of degree $4k$, e.g., $\sum_{j=0}^3 x_j^{4k}$.
If $k>1$, then $M_{4k}$ is a minimal K\"{a}hler surfaces of general type. 
Define the antiholomorophic free involution $\iota$  by 
\[
[x_0,x_1,x_2,x_3]\mapsto [\overline{x}_1,-\overline{x}_0,\overline{x}_3,-\overline{x}_2].
\]
Let $\hat{M}_{4k}=M_{4k}/\iota$.
We check the assumptions.
\begin{Lemma}\label{lem:w2}
There is a lift of $w_2(\hat{M}_{4k})$ in the torsion part of $H^2(\hat{M}_{4k};\Z)$. 
\end{Lemma}
\begin{proof}
See \cite{Habegger}. 
The proof in \cite{Habegger} is on $\hat{M}_4$, but it works well for  $\hat{M}_{4k}$.
\end{proof}
\begin{Proposition}\label{prop:w2}
$w_2(\hat{K})=0$ and $w_2(\hat{M}_{4k}) = w_1(\ell_{\R})^2$.  
$\pi^*\colon H^1(\hat{M}_{4k};\Z_2)\to H^1(M_{4k};\Z_2)$ is surjective.
\end{Proposition}
\begin{proof}
The fact that $w_2(\hat{K})=0$ follows from that the canonical bundle  $K$ of $M_{4k}$  is given by
\[
K  = (4k-4)H,
\]
where $H$ is the hyperplane section.
By \lemref{lem:w2} and \cite[\S1, Remark 3(2)]{N1}, there exists a class $\alpha\in H^1(\hat{M}_{4k};\ell)$ such that $ w_2(\hat{M}_{4k})=\alpha\cup\alpha$. 
Since $\pi_1(M_{4k})=1$, $\pi_1(\hat{M}_{4k})=\Z/2$, $\alpha$  must be $ w_1(\ell_{\R})$ and $\pi^*$ is surjective. 
\end{proof}
\begin{Corollary}
There exists a canonical $\Spincm$ structure $\mathfrak{s}_0$ on $M_{4k}\to \hat{M}_{4k}$.
\end{Corollary}
\begin{Proposition}
The moduli space for $(\hat{M}_{4k},\hat{\mathfrak{s}}_0)$ is orientable and its virtual dimension is $0$. 
Therefore the $\Pin^-(2)$-monopole invariant of $(\hat{M}_{4k},\hat{\mathfrak{s}}_0)$ can be defined as a $\Z$-valued invariant.
\end{Proposition}
\begin{proof}
Note $b_1^\ell(\hat{M}_{4k})=0$.
In order to prove the orientability of the moduli space, it suffices to prove the Dirac index, $\ind D$, of $\hat{\mathfrak{s}}_0$ is even by \cite[Proposition 2.15]{N2}. 
Let $d(\hat{\mathfrak{s}}_0)$ be the virtual dimension of the moduli space.
Since $\iota$ is free, we have
$d(\hat{\mathfrak{s}}_0) = \frac12d(\mathfrak{s}_0) = 0$.
Then 
\[
0=d(\hat{\mathfrak{s}}_0)=\ind D - (b_0^\ell-b_1^\ell +b_+^\ell).
\]
Since $\ell$ is nontrivial, $b_0^\ell=0$.
Therefore
\[
\ind D=b_+^\ell = \frac12(1+b_+(M_{4k})) = \frac12\left(\frac{4k}{3}\{16k^2 -24k +11\}\right).
\]
(For the calculation of $b_+(M_{4k})$, see e.g.\cite[Example 4.27]{MS}.)
Therefore $\ind D$ is even.
\end{proof}
\begin{Remark}
Note that $\hat{M}_{4}$ is diffeomorphic to an Enriques surface.
If $k>1$, then all of the ordinary Seiberg--Witten invariants of $\hat{M}_{4k}$ are zero by a theorem due to S.~Wang~\cite{wang95van}.
\end{Remark}
%
%
\subsection{Elliptic surfaces}\label{subsec:elliptic}
%
%
In this subsection, we compute the $\Pin^-(2)$-monopole invariants of the quotient manifolds of  some elliptic surfaces.
First, we construct anti-holomorphic involutions on certain elliptic surfaces over $\CP^1$.

A method to construct elliptic fibrations by using hyperelliptic involutions is given in Gompf--Stipsicz's book \cite{GS}, \S3.2.
Let $\Sigma_k$ be a Riemannian surface of genus $k$, and $h_k\colon\Sigma_k\to\Sigma_k$ be a hyperelliptic involution.
Take the diagonal $\Z_2$-action $h_k\times h_1$ on $\Sigma_k\times \Sigma_1$.
Dividing by the $\Z_2$-action $h_k\times h_1$, we obtain the quotient $(\Sigma_k\times\Sigma_1)/\Z_2$ with $4(2k+2)$ singular points.
Resolving the singular points produces a complex manifold $X(k+1)$.
Dividing the projection $\mathrm{pr}_1\colon\Sigma_k\times\Sigma_1\to \Sigma_k$ and extending it to the resolution, we obtain the elliptic fibration $\varpi\colon X(k+1)\to\CP^1$.
It is well-known that $X(n)$ is diffeomorphic to $E(n)$, the fiber sum of $E(1)=\CP^2\#9\overline{\CP}^2$.

We construct an anti-holomorphic free involution on $X(2n)$.
Take the antipodal map $\iota_0$ on $\CP^1=\C\cup\{\infty\}$ defined by $z\mapsto z^*:=-1/\bar{z}$. 
Choose $k$ distinct points $a_1,\ldots,a_k$ on $\CP^1$ satisfying $0<|a_i|<1$.
Let $\Sigma_k$ be the hyperelliptic curve defined by the equation
\[
w^2 = z(z-a_1)(z-a_1^*)\cdots(z-a_k)(z-a_k^*),
\]
and $\Sigma_k\to\CP^1$ the associated double covering branched at $a_1$, $a_1^*,\ldots, a_k$, $a_k^*$, $0$, $\infty$. 
Then the antipodal map $\iota_0$ on the base $\CP^1$ can be lifted to an anti-holomorphic map $\sigma_k$ on $\Sigma_k$ with order $2$ if $k$ is odd, and with order $4$ if $k$ is even.

Suppose $k=2n-1$ for a positive integer $n$.
Take the diagonal action $\sigma_{2n-1}\times \sigma_1\colon \Sigma_{2n-1}\times\Sigma_1\to \Sigma_{2n-1}\times\Sigma_1$.
Then $\sigma_{2n-1}\times \sigma_1$ descends to a free involution on the quotient $(\Sigma_{2n-1}\times\Sigma_1)/\Z_2$. 
Furthermore we can easily extend it to an anti-holomorphic free involution $\iota$ on $X(2n)$ which covers the antipodal map $\iota_0$ on the base $\CP^1$.
\begin{Proposition}
The surface $X(2n)$ admits a K\"{a}hler form $\omega$ such that $\iota^*\omega=-\omega$.
\end{Proposition}
\begin{proof}
We can easily construct a K\"{a}hler form $\omega_0$ on $\Sigma_k\times\Sigma_1$ such that $(h_k\times h_1)^*\omega_0=\omega_0$ and $(\sigma_k\times \sigma_1)^*\omega_0=-\omega_0$.
Then $\omega_0$ induces a singular K\"{a}hler form $\hat{\omega}_0$ on $(\Sigma_k\times\Sigma_1)/\Z_2$.
By the results due to Fujiki \cite{Fujiki}, we can obtain a K\"{a}hler form $\omega$ on $X(2n)$. 
Moreover we can choose $\omega$ such that $\iota^*\omega=-\omega$.
\end{proof}

Let $\X(2n)=X(2n)/\iota$. 
By construction, $\varpi\colon X(2n)\to\CP^1$ descends to $\hat{\varpi}\colon \X(2n)\to\RP^2$. 
The general fiber of $\hat{\varpi}$ is a torus. 
Note that $X(2)=E(2)$ is a $K3$ surface. 
\begin{Proposition}
The quotient manifold $\X(2)=X(2)/\iota$ is diffeomorphic to an Enriques surface.
\end{Proposition}
\begin{proof}(\cite{Donaldson} and
\cite{realenriques}, \S15.1.)
Take an $I$-invariant holomorphic form $\phi$ on $X(2)$.
By the Calabi-Yau theorem, there exists a unique K\"{a}hler-Einstein metric.
Then $\phi$ and $\omega$ induce a hyper-K\"{a}hler structure on $X(2)$.
There exists a complex structure for which $\iota$ is a holomorphic free involution.
Thus $\X(2)$ is an Enriques surface.
\end{proof}

Since $X(2n)$ is diffeomorphic to the fiber sum of $E(2)=K3$ with $E(2n-2)$, $\X(2n)$ is diffeomorphic to the fiber sum of the fibration $\hat{\varpi}\colon \X(2)\to \RP^2$ with $E(n-1)$.

\begin{Proposition}
If $k\equiv 2$ modulo $4$, then there exists a canonical $\Spincm$ structure $\hat{\mathfrak{s}}_0$ on $X(k)\to\X(k)$.
\end{Proposition}
\begin{proof}
Since $\X(4m+2)$ is the fiber sum of $\hat{\varpi}\colon \X(2)\to \RP^2$ and $E(2m)$,
it is easy to see that $\X(4m+2)$ is a non-spin manifold with $\pi_1(\X(4m+2))=\Z/2$ whose intersection form is isomorphic to 
\[
(2m+1)(-E_8)\oplus (4m+1) H, 
\]
where $H$ is a hyperbolic form.
Then it follows from a result of Hambleton-Kreck \cite{HK} ({\it Cf.} \cite{Ue}) that $\X(4m+2)$ is homeomorphic to the connected sum
\[
\Sigma\# (2m+1)|E_8| \# (4m+1)(S^2\times S^2),
\] 
where $\Sigma$ is a rational homology $4$-sphere such that $\pi_1(\Sigma)=\Z/2$ and $w_2(\Sigma)\neq 0$, and $|E_8|$ is the $E_8$-manifold, i.e., the simply-connected topological manifold whose intersection form is isomorphic to $-E_8$.
Since $|E_8|$ and $S^2\times S^2$ are spin and $w_2(\Sigma)$ is a torsion class,  $w_2(\hat{X}(4m+2))$ is a torsion class.

Since the canonical divisor $K$ of $X(n)$ is 
$K=(n-2) F$, where $F$ is a general fiber, we can see that $\tilde{c}_1(\hat{K})$ is divided by $2$ if $n=4m+2$.
The rest of the proof is similar to that of \propref{prop:w2}.
\end{proof}
Take an $I$-invariant divisor $D_k$ of $X(n)$ of the form
\[
D_k=\sum_{i=1}^k (F_i+IF_i),
\]
where $F_i$ are general fibers.
Let $E_k$ be the line bundle associated to $D_k$.
Then $E_k$ can be written as the pull-back $E_k=\varpi^*L$ where $\varpi\colon X(n)\to\CP^1$ is the elliptic fibration and $L$ is a line bundle over $\CP^1$ of degree $2k$.
Let $\hat{E}_k=E_k/I$. 

The next is an analogue of \cite[Proposition 4.2]{FM2} or \cite[Proposition 42]{Brussee}.
\begin{Theorem}\label{thm:elliptic}
Let $\X=\X(4m+2)$ and $\hat{\mathfrak{s}}_k=\hat{\mathfrak{s}}_0\hat{\otimes}\hat{E}_k$.
The moduli space $\hat{\M}(\X,\hat{\mathfrak{s}}_k)$ is orientable and the corresponding invariant is 
\[
\SW^{\Pin}(\X,\hat{\mathfrak{s}}_k) = \pm
\begin{pmatrix}
2m\\k
\end{pmatrix}.
\]
\end{Theorem}
The rest of this section is devoted to the proof of \thmref{thm:elliptic}.
For $q\in\CP^1$, let $q^*=\iota_0(q)$ where $\iota_0$ is the antipodal map. 
Choose $2m$ distinct points $q_1,\ldots,q_{2m}$ on $\CP^1$ such that all of $q_1,\ldots,q_{2m}$ and $q_1^*,\ldots,q_{2m}^*$ are distinct.
Let $F_i=\varpi^{-1}(q_i)$.
Then $IF_i=\varpi^{-1}(q_i^*)$, and  we obtain an $I$-invariant canonical divisor $D_m=\sum_{i=1}^{2m}(F_i+IF_i)$ of $X=X(4m+2)$. 
Let $\eta$ be the corresponding $I$-invariant holomorphic section on the canonical bundle $\mathscr{K}$.
By \corref{cor:isom},  the $\Pin^-(2)$-monopole moduli space  $\M(\X,\hat{\mathfrak{s}}_k)$ is identified with 
\[
\mathcal{V}_t(E,E^\prime,\Sigma,\eta)^I \cong \mathscr{M}^s(E,E^\prime,\mathscr{K},\eta)^I\cong \mathscr{D}_b(\Delta)^I,
\]
where $E^\prime = K\otimes E^{-1}$ and $\Delta= Z(\eta)$.
\begin{Lemma}
The moduli space $\M(\X,\hat{\mathfrak{s}}_k)$ is $0$-dimensional and orientable.
\end{Lemma}
\begin{proof}
Let $\mathfrak{s}_k$ be the $\Spinc$ structure on $X$ of the canonical reduction of $\pi^*\hat{\mathfrak{s}}_k$.
Note that  $c_1(L)^2=0$, $\tau(X)=-16(2m+1)$ and $2e(X)+3\tau(X)=0$ , where $L=K^{-1}\otimes E_k^2$ is the determinant line bundle, and  $e(X)$ and $\tau(X)$ are the Euler characteristic and signature of $X$.
Then the virtual dimension $d(\mathfrak{s}_k)$ of the Seiberg--Witten moduli space of $(X,\mathfrak{s}_k)$ is 
\[
d(\mathfrak{s}_k) = \frac14(c_1(L)^2-2e(X)-3\tau(X))=0.
\] 
Then we have $d(\hat{\mathfrak{s}_k})=d(\mathfrak{s}_k)/2 =0$.

Since the index of the Dirac operator $D_{\hat{A}}$ on $\hat{\mathfrak{s}}_k$ is a half of that on $\mathfrak{s}$, we have
\[
\ind D_{\hat{A}} = \frac12\left(\frac14(c_1(L)^2-\tau(X))\right) = 4m+2.
\]
Especially, $\ind D_{\hat{A}}$ is even.
Then  the moduli space is orientable by  \cite[Proposition 2.15]{N2}.
\end{proof}
\begin{Proposition}
Every $\Pin^-(2)$-monopole solution corresponding to a divisor $D\in \mathscr{D}_b(\Delta)^I$ is non-degenerate.
\end{Proposition}
\begin{proof}
Since $D_m-D$ and $D$ have no intersection for $D\in \mathscr{D}_b(\Delta)^I$, 
\propref{prop:maxdiv} implies that the first cohomology $H^1$ of the deformation complex of the solution corresponding to $D$ is $0$.
Therefore the second cohomology $H^2$ is also $0$ since $d(\hat{\mathfrak{s}_k}) =0$ and $H^1=0$.  
\end{proof}
For a subset $\{i_1,\ldots,i_k\}\subset \{1,\ldots,2m\}$, we have an $I$-invariant divisor $\sum_{j=1}^k (F_{i_j} + IF_{i_j})  \in\mathscr{D}_b(\Delta)^I$. Since this correspondence is bijective, the number of elements in  $\mathscr{D}_b(\Delta)^I$ is 
$\begin{pmatrix}
2m\\k
\end{pmatrix}.
$
In order to complete the proof of  \thmref{thm:elliptic}, we show that all of the divisors in $\mathscr{D}_b(\Delta)^I$ have same orientation.

\begin{Proposition}\label{prop:OXD}
For distinct points $b_1,\ldots,b_k\in\CP^1$, consider the divisor $B=b_1+\cdots+b_k$.
Let $F_j=\varpi^*(b_j)$ and  $D=F_1+\cdots +F_j$.
Then $\varpi^*\colon H^0(\mathscr{O}_{\CP^1}(B))\to H^0(\mathscr{O}_{X}(D))$ is an isomorphism, and  
\[
H^0(\mathscr{O}_X(D))=\C^{k+1},\quad H^1(\mathscr{O}_X(D))=0,\quad H^2(\mathscr{O}_X(D))=\C^{4m-k}
\]
\end{Proposition}
\begin{proof}
Consider the short exact sequence 
\[
0\to\mathscr{O}_X\to\mathscr{O}_X(D)\to\mathscr{O}_D(D)\to0
\]
and its associated long exact sequence.
Note that $\mathscr{O}_{F_j}(D)$ is the holomorphic normal bundle of $F_j\subset X$.
Since this is holomorphicallly trivial and the genus of $F_j$ is $1$, we have
\[
H^0(\mathscr{O}_D(D))=H^1(\mathscr{O}_D(D))=\C^k.
\]
Since $X$ is simply-connected, $H^1(\mathscr{O}_X)=0$ and therefore
\[
H^1(\mathscr{O}_X(D))=H^0(\mathscr{O}_X)\oplus H^1(\mathscr{O}_D(D))=\C^{k+1}.
\]
By the Serre duality,
\[
H^2(\mathscr{O}_X(D)) = H^0(\mathscr{O}_X(K-D)) = \C^{4m-k}.
\]
Also we have $H^1(\mathscr{O}_X(D))=0$.
\end{proof}

Let $\mathfrak{t}=(B,\phi,\psi)\in \mathcal{A}(E,h)\times \Omega^{0,0}(E)\times \Omega^{2,0}(E^{-1})$ be an $I$-invariant solution to \eqref{eq:WV}, i.e., $((B,\phi),(B^\prime,\phi^\prime))$ where $B^\prime = \Sigma \otimes B^{-1}$ and $\phi^\prime = \psi$  is $I$-invariant and satisfies \eqref{eq:WV}.
The deformation complex at $\mathfrak{t}$ is given by
\begin{gather*}
\left(i\Omega^0_X\right)^I\overset{D_{\mathfrak{t}}^0}{\to}\left(i\Omega^1_X\oplus\Omega^{0,0}(E)\oplus\Omega^{2,0}(E^{-1})\right)^I\overset{D_{\mathfrak{t}}^1}{\to}\left(\Omega^{0,1}(E)\oplus i \left( \Omega_X^{0,2}\oplus\Omega^{1,1}_X\right)\right)^I,\\
D_{\mathfrak{t}}^0(f)=\begin{pmatrix}df\\ f\phi\\ -f\psi\end{pmatrix},\quad
D_{\mathfrak{t}}^1(\dot{B},\dot\phi, \dot\psi)=\begin{pmatrix}
\dot{B}^{0,1}\phi + w(\dot{B}^{0,1})^*\bar\psi + \bar\partial_B\dot\phi + \bar\partial^*_B \bar{\dot\psi}\\
\bar\partial \dot{B}^{0,1} - \frac12\bar\psi\bar{\dot\phi} - \frac12 \bar{\dot\psi}\bar{\phi}\\
\left\{\Lambda_g d \dot{B} -i[\mathrm{Re} (\dot\phi\bar\phi)-\mathrm{Re}(\bar\psi\dot\psi)]\right\}\omega 
\end{pmatrix},
\end{gather*}
where $w(\dot{B}^{0,1})^*$ is the adjoint of the multiplication operator $\dot{B}^{0,1}\wedge\cdot$.
The adjoint of $D_{\mathfrak{t}}^0$ is 
\[
\left(D_{\mathfrak{t}}^0\right)^*(\dot{B},\dot\phi, \dot\psi) = d^*\dot{B} - i\Ima(\dot\phi\bar\phi)+ i\Ima(\bar\psi\dot\psi).
\]
Set $\dot{b}=\sqrt2\dot{B}$, $U=\frac1{\sqrt2}\phi$, $V=\frac1{\sqrt2}\psi$. 
Replace $\dot{b}$ by $\dot{b}^{0,1}$ via the identification $i\Omega_X\cong \Omega^{0,1}$.
We introduce the operators $Q_{\mathfrak{t}}$ and $Q_{\mathfrak{t}}^0$ by
\begin{gather*}
Q_{\mathfrak{t}} = \begin{pmatrix}
-\id & 0 & 0& 0\\
0& \sqrt2\id & 0&0\\
0&0& \frac1{\sqrt2 i \omega} &0\\
0&0&0&\frac1{\sqrt2}
\end{pmatrix}
\begin{pmatrix}
D_{\mathfrak{t}}^1\\\left(D_{\mathfrak{t}}^0\right)^*
\end{pmatrix},\\
Q_{\mathfrak{t}}\begin{pmatrix}
\dot{b}\\ \dot\phi\\ \dot\psi
\end{pmatrix}
= \begin{pmatrix}
-b^{0,1}U - w(\dot{b}^{0,1})^* \bar{V} -\bar{\partial}_B\dot\phi-\bar{\partial}_B^*\bar{\dot\psi}\\
\bar{\partial}\dot{b}^{0,1} - \bar{V}\bar{\dot\phi} -\bar{\dot\psi}\bar{U} \\
\bar{\partial}^*\dot{b}^{0,1} - \dot\phi\bar{U} +\bar{V}\dot\psi 
\end{pmatrix},\\
Q_{\mathfrak{t}}^0\begin{pmatrix}
\dot{b}\\ \dot\phi\\ \dot\psi
\end{pmatrix}
= \begin{pmatrix}
 -\bar{\partial}_B\dot\phi-\bar{\partial}_B^*\bar{\dot\psi}\\
\bar{\partial}\dot{b}^{0,1}  \\
\bar{\partial}^*\dot{b}^{0,1} 
\end{pmatrix}.
\end{gather*}
Let $K_0$ and $C_0$ be the kernel and cokernel of $Q^0_{\mathfrak{t}}$, respectively,
\begin{align*}
K_0:=&\ker Q^0_{\mathfrak{t}} = \ker (\bar{\partial}_B+\bar{\partial}_B^*)^I\oplus \mathbb{H}^1(\mathscr{O}_X)^I,\\
C_0:= &\coker Q^0_{\mathfrak{t}} = \ker (\bar{\partial}^*_B\oplus\bar{\partial}_B)^I\oplus \mathbb{H}^2(\mathscr{O}_X)^I\oplus \mathbb{H}^0(\mathscr{O})^I.
\end{align*}
Note that $\mathbb{H}^1(\mathscr{O}_X)=0$ since $b_1(X)=0$.
By \propref{prop:OXD}, $\ker (\bar{\partial}_B^*\oplus\bar{\partial}_B) = H^1(\mathscr{O}_X(D))=0$. 
%

To see the orientation of the solution $\mathfrak{t}$, we consider
\[
\mathrm{pr}_{C_0}\circ (Q_{\mathfrak{t}}|_{K_0})\colon K_0=\ker (\bar{\partial}_B+\bar{\partial}_B^*)^I\to C_0=H^2(\mathscr{O}_X)^I\oplus H^0(\mathscr{O}_X)^I.
\]
This can be identified with
\begin{gather*}
R_{\mathfrak{t}}\colon H^0(D)^I\oplus H^0(K-D)^I\to H^0(K)^I\oplus H^0(\mathscr{O}_X)^I,\\
R_{\mathfrak{t}}=\begin{pmatrix}
\dot\phi\\ \dot\psi
\end{pmatrix}
=\begin{pmatrix}
-V\dot\phi-U\dot\psi\\
\int(-\langle\dot\phi,U\rangle+\langle\dot\psi,V\rangle)d\mathrm{vol}_g
\end{pmatrix}.
\end{gather*}
Note that $R_{\mathfrak{t}}$ is a linear isomorphism.
If we fix orientations of the domain and target of $R_{\mathfrak{t}}$, the orientation of the solution $\mathfrak{t}$ is determined by the sign of the determinant of $R_{\mathfrak{t}}$.
(See e.g. \cite{Teleman},  \cite{Nicolaescu}.)
We want to represent $R_{\mathfrak{t}}$ by a matrix with some explicit bases of $H^0(D)^I\oplus H^0(K-D)^I$ and $H^0(K)^I\oplus H^0(\mathscr{O}_X)^I$.

Consider a complex manifold $Z$ with a divisor $D$.
Let $\mathscr{L}_D$ be the holomorphic line bundle associated with $D$.
Then $H^0(\mathscr{L}_D)$ can be identified with the space of meromorphic functions $f$ such that $\lambda f$ are holomorphic for every local defining function $\lambda$ of $D$.
This space is denoted by $\mathfrak{M}(D)$.
Note that $1\in\mathfrak{M}(D)$  corresponds to the holomorphic section $\phi_D$ defined by the divisor $D$.

Suppose $\iota$ is an anti-holomorphic free involution on $Z$ and the divisor $D$ is $I$-invariant. 
The $I$-action on $H^0(\mathscr{L}_D)=\mathfrak{M}(D)$ is given by $f\mapsto \overline{\iota^*f}$.  


Consider $(Z,\iota)=(\CP^1,\iota_0)$  where $\iota_0$ is the antipodal map.
Choose $p\in\C\subset \CP^1$ such that $|p|=1$.
Let $p^*=\iota_0 (p)=-1/\bar{p} = -p$.
Take the $\iota_0$-invariant divisor $B=p + p^*$.
In terms of meromorphic functions, we can take the following for a complex basis of   $\mathfrak{M}(B)=H^0(\mathscr{L}_B)$.
\[
1,\quad \frac{z-p^*}{z-p},\quad \frac{z-p}{z-p^*}.
\]
We want to have a (real) basis for $\mathfrak{M}(B)^I=H^0(\mathscr{L}_B)^I$.
Via the projection to the $I$-invariant part, $f\mapsto \frac12(f+If)$, we  see that the following is a real basis for $\mathfrak{M}(B)^I=H^0(\mathscr{L}_B)^I$:
\[
1,\quad i\frac{z^2+p^2}{z^2-p^2},\quad \frac{2pz}{z^2-p^2}.
\]

Let us consider the case of $\varpi\colon X=X(4m+2)\to \CP^1$ with the antiholomorphic involution $\iota$.
Choose $2m$ distinct points $p_1,\ldots,p_k$, $q_1,\ldots,q_{2m-k}$ on $\CP^1$ such that $|p_j|=1$, $|q_l|=1$, and all of $p_j$, $p_j^*$, $q_l$, $q_l^*$ are distinct. 
Take the following divisors on $\CP^1$:
\[
B_D=\sum_{j=1}^k(p_j+p^*_j),\quad B_{K-D} = \sum_{l=1}^{2m-k}(q_l+q_l^*),\quad B_{K}=B_D+B_{K-D}.
\]
Then $K=\varpi^{*}B_K$ is a canonical divisor of $X$.
Let $D=\varpi^{*}B_D$. Then $K-D=\varpi^{*}B_{K-D}$.
Let
\[
P_j^1 =  i\frac{z^2+p_j^2}{z^2-p_j^2},\quad P_j^2 =  \frac{2p_jz}{z^2-p_j^2},\quad Q_l^1 =  i\frac{z^2+q_l^2}{z^2-q_l^2},\quad Q_l^2 =  \frac{2q_lz}{z^2-q_l^2}.
\]
Then $\{1,P_j^1,P_j^2\}$ ($j=1,\ldots,k$) is  a basis of $\mathfrak{M}(B_D)^I = H^0(B_D)^I\cong H^0(D)^I$.
Similarly, $\{1,Q_l^1,Q_l^2\}$  ($l=1,\ldots,2m-k$) and $ \{1,P_j^1,P_j^2,Q_l^1,Q_l^2\}$ ($j=1,\ldots,k$, $l=1,\ldots,2m-k$) are bases of $\mathfrak{M}(B_{K-D})^I = H^0(B_{K-D})^I\cong H^0(K-D)^I$ and $\mathfrak{M}(B_K)=H^0(B_K)^I\cong H^0(K)^I $, respectively.

Now the divisor $D$ corresponds to an $I$-invariant solution $\mathfrak{t}=(B,\phi,\psi)$ such that $\phi^{-1}(0)=D$ and $\psi^{-1}(0)=K-D$.
We may assume the following correspondence,
\begin{align*}
1\in \mathfrak{M}(B_D)^I &\longleftrightarrow U=\frac1{\sqrt2}\phi\in H^0(D)^I\\
1\in \mathfrak{M}(B_{K-D})^I &\longleftrightarrow V=\frac1{\sqrt2}\psi\in H^0(K-D)^I\\
1\in \mathfrak{M}(B_K)^I &\longleftrightarrow W:=U\otimes V \in H^0(K)^I
\end{align*}
and, for $f\in\mathfrak{M}(D)$, let $fU$ denote the  holomorphic section in $H^0(D)$ corresponding to $f$.

Let $\{e\}$ be a real basis of $H^0(\mathscr{O}_X)^I\cong\R$.
We choose the basis  for $H^0(D)^I\oplus H^0(K-D)^I$  as follows, 
\begin{equation}\label{eq:basis1}
\{U,P_j^1U,P_j^2U\}_{j=1,\ldots,k}\cup \{V,Q_l^1V,Q_l^2V\}_{l=1,\ldots,2m-k},
\end{equation}
and for $H^0(K)^I \oplus H^0(\mathscr{O}_X)^I$,
\begin{equation}\label{eq:basis2}
\{W,P_j^1W,P_j^2W,Q_l^1W,Q_l^2W\}_{j=1,\ldots,k}^{ l=1,\ldots,2m-k}\cup\{e\},
\end{equation}
With respect to the above bases, the isomorphism $R_{\mathfrak{t}}$ is represented by the matrix 
{\small
\[
\bordermatrix{&U&P^1_1U&P^2_1U&\cdots & V & Q^1_1V & Q^2_1V &\cdots\cr
W&-1& 0 &0&\cdots& -1 &0&0&\cdots\cr 
P^1_1W&0& -1 &0&\cdots& 0 &0&0&\cdots\cr 
P^2_1W&0& 0 &-1&\cdots& 0 &0&0&\cdots\cr 
\vdots & \vdots & \vdots & \vdots & \ddots & \vdots & \vdots & \vdots &\ddots \cr
Q^1_1W & 0 & 0 &0 & \cdots &0 & -1 & 0 & \cdots\cr
Q^2_1W & 0 & 0 &0 & \cdots &0 & 0 & -1 & \cdots\cr 
\vdots & \vdots & \vdots & \vdots & \ddots & \vdots & \vdots & \vdots &\ddots \cr
e & -\|U\|^2 & -\langle U,P_1^1U\rangle & -\langle U,P_1^2U\rangle & \cdots & \|V\|^2 & \langle V,Q_1^1V\rangle & \langle V,Q_1^2V\rangle &\cdots
}
\] }
It is easy to see that the determinant of the above matrix is $-\|U\|^2 -\|V\|^2$.

In order to prove any other solution $\mathfrak{t}^\prime$ has the same orientation with the solution $\mathfrak{t}$ for $D$ above, we need to prove that the determinant of the matrix $R_{\mathfrak{t}^\prime}$ associated with $\mathfrak{t}^\prime$ has the same sign with $\det R_\mathfrak{t}$.

For simplicity, we consider the case when $m=2$, $k=1$.
The generalization to other cases will be obvious. 
Take the solution $\mathfrak{t}$ corresponding to the divisors  
$D=\varpi^{*}(p_1+p_1^* + p_2+p_2^*)$ and $K-D=\varpi^{*}(q_1+q_1^* + q_2+q_2^*)$. 
Exchanging $p_1$ and $q_1$, we obtain another $I$-invariant solution $\mathfrak{t}^\prime$ corresponding to $D^\prime = \varpi^{*}(q_1+q_1^* +p_2+p_2^*)$ and $K-D^\prime = \varpi^{*}(p_1+p_1^*+q_2+q_2^*)$.
Let $U$, $V$ be the holomorphic sections for $\mathfrak{t}$, and $U^\prime$, $V^\prime$ for $\mathfrak{t}^\prime$.
Without loss of generality, we may assume $U^\prime$, $V^\prime$ are related with $U$, $V$ by
\[
U^\prime = \frac{p_1}{q_1}\frac{z^2-q_1^2}{z^2-p_1^2} U, \quad V^\prime = \frac{q_1}{p_1}\frac{z^2-p_1^2}{z^2-q_1^2} V.
\] 
The meaning of this  is as follows:
{\it When  the holomorphic section $U^\prime$ associated with the solution $\mathfrak{t}^\prime$   is considered as an element of $H^0(D)^I = \mathfrak{M}(D)^I$ {\rm (}not of $\mathfrak{M}(D^\prime)^I${\rm)}, it is represented by the meromorphic function $ \frac{p_1}{q_1}\frac{z^2-q_1^2}{z^2-p_1^2}$.}

We want to represent the map $R_{\mathfrak{t}^\prime}$
\[
R_{\mathfrak{t}^\prime}\begin{pmatrix}
\dot\phi\\ \dot\psi
\end{pmatrix}
=\begin{pmatrix}
-V^\prime\dot\phi-U^\prime\dot\psi\\
\int(-\langle\dot\phi,U^\prime\rangle+\langle\dot\psi,V^\prime\rangle)d\mathrm{vol}_g
\end{pmatrix}
\]
by a matrix with respect to the bases \eqref{eq:basis1}, \eqref{eq:basis2}.

Before that we note several useful relations.
For $p_j$ and $q_l$ with $|p_j|=|q_l|=1$, let $a$, $b$ be the real numbers such that 
$\dfrac{p_j}{q_l} = a + ib$.
Then we have the following relations.
\begin{align*}
P^1_jQ^1_l &= -1 + \frac{a}{b} P^1_j-\frac{a}{b}Q^1_l\\
P^1_jQ^2_l &=  \frac{1}{b} P^2_j-\frac{a}{b}Q^2_l\\
P^2_jQ^1_l &= - \frac{1}{b} Q^2_l+\frac{a}{b}P^2_j\\
P^2_jQ^2_l &= - \frac{1}{b} P^1_j+\frac{1}{b}Q^1_l
\end{align*}


Let $a,b,c,d,e,f$ be the real numbers such that 
\[
\frac{p_1}{q_1}=a+ib,\quad \frac{p_2}{q_1} = c+id,\quad \frac{p_1}{q_2} = e+if.
\]
Then we have $U^\prime = (a+bP^1_1)U$, $V^\prime = (a-bQ^1_1)V$, and the map $R_{\mathfrak{t}^\prime}$ is represented by the matrix
{\small
\[
\bordermatrix{&U&P^1_1U&P^2_1U&P^1_2U&P^2_2U & V & Q^1_1V & Q^2_1V &Q^1_2V & Q^2_2 V\cr
W&-a&-b& &-b& & -a& b &&b&\cr 
P^1_1W&&&&&&-b&-a&&-\frac{be}{f}&\cr 
P^2_1W&&&&&&&&-1&&-\frac{b}f\cr 
P^1_2W&&&&-\frac{ad-bc}{d}&&&&&&\cr 
P^2_2W&&&&&-\frac{ad-bc}{d}&&&&&\cr 
Q^1_1W &b&-a&&\frac{bc}{d}&&&&&&  \cr
Q^2_1W & &&-1&&-\frac{b}{d}&&&&&\cr 
Q^1_2W & &&&&&&&&-\frac{af-be}{f}&\cr
Q^2_2W & &&&&&&&&&-\frac{af-be}{f}\cr 
e & -u_0& -u_1^1 & -u^2_1 &-u^1_2 &-u^2_2 & v_0 & v^1_1 & v^2_1 &v^1_2 & v^2_2 }
\]}
where $u_0=\langle U,U^\prime\rangle$, $u^a_j=\langle P^a_jU,U^\prime\rangle$,
$v_0=\langle V,V^\prime\rangle$, $v^b_l=\langle Q^b_lV,V^\prime\rangle$.
Furthermore, we take one more {\it orientation-preserving} basis change given by
\begin{align*}
(U,P^1_1U)\begin{pmatrix} a & -b \\ b &a\end{pmatrix} &= (U^\prime,(-b+aP^1_1)U),\\
(V,Q^1_1V)\begin{pmatrix} a & b \\ -b &a\end{pmatrix} &= (V^\prime,(b+aQ^1_1)V).
\end{align*}
Then the matrix above is transformed into 
{\small
\[
\begin{pmatrix}
-1&& &-b& & -1&  & & b & \\
&&&&&&-1&&-\frac{be}{f}&\\ 
&&&&&&&-1&&-\frac{b}f\\
&&&-\frac{ad-bc}{d}&&&&&&\\ 
&&&&-\frac{ad-bc}{d}&&&&&\\ 
&-1&&\frac{bc}{d}&&&&&&  \\
 &&-1&&-\frac{b}{d}&&&&&\\
 &&&&&&&&-\frac{af-be}{f}&\\
 &&&&&&&&&-\frac{af-be}{f}\\ 
 -\|U^\prime\|^2& \ast & \ast & \ast &\ast & \|V^\prime\|^2 & \ast & \ast & \ast & \ast
\end{pmatrix}
\]}
where $\ast$ are some numbers of inner products of sections.
It is easy to see that the determinant of the above matrix is 
\[
-\left(\frac{ad-bc}{d}\right)^2\left(\frac{af-be}{f}\right)^2(\|U^\prime\|^2+\|V^\prime\|^2).
\]
Since $\det R_{\mathfrak{t}}$ and $\det R_{\mathfrak{t}^\prime}$ have the same sign, the orientation of the solution $\mathfrak{t}^\prime$ is same with that of $\mathfrak{t}$.

The above proof is easily generalized to the cases when $k\neq 1$ or $m\neq 2$, and we see that the orientation does not change by an exchange of some $p_j$ and $q_l$.
Thus \thmref{thm:elliptic} is proved.
%
%
\section{Concluding remarks}\label{sec:remarks}
%
%
\subsection{More examples with nontrivial $\Pin^-(2)$-monopole invariants}
%
By using the gluing formulae in \cite{N2}, we obtain more examples with nontrivial $\Pin^-(2)$-monopole invariants.
Let $Z\to \hat{Z}$ be a nontrivial double covering which satisfies the following:
\begin{enumerate}
\item $b_+^{\ell^\prime}(\hat{Z})=0$ for $\ell^\prime=Z\times_{\deux}\Z$.
\item There is a $\Spincm$ structure $\mathfrak{s}^\prime$ on $Z\to\hat{Z}$ whose characteristic bundle $\hat{L}^\prime$ satisfies $\tilde{c}_1(\hat{L}^\prime)^2 = \sign(Z)$.
\end{enumerate} 
(For instance, a connected sum of several $S^2\times\Sigma_g$ and $S^1\times W$ has a double cover satisfying the above conditions \cite[\S1.2]{N2}, where $\Sigma_g$ is a Riemann surface with genus $g\geq 1$ and $W$ is a closed $3$-manifold.)
Then \cite[Theorem 3.11]{N2} implies that $\hat{M}_{4k}\#\hat{Z}$ and $\X(4m+2)\#\hat{Z}$ has nontrivial $\Pin^-(2)$-monopole invariants.

Furthermore, \cite[Theorem 3.13]{N2} implies that any connected sum $\hat{Y}_1\#\cdots\#\hat{Y}_N$ such that each $\hat{Y}_i$ is $\hat{M}_{4k}$ or $\X(4m+2)$ for any $k$ or $m$ has nontrivial $\Pin^-(2)$-monopole invariants.

As an application of the nontriviality of the $\Pin^-(2)$-monopole invariants, we have the adjunction inequality for local-coefficient classes \cite[Theorem 1.15]{N2}.
%
\subsection{Problems}
%
We suggest several problems for future researches.
\begin{itemize}
\item Generalize the results to the case of the real structures {\it with real parts}.
If we can drop the condition that $\iota$ is {\it free} in our story, then we might expect some applications to, say,  real algebraic geometry.
For this purpose, we need to generalize the notion of the $\Spincm$-structure.
\item Analogy of SW=Gr \cite{Taubes4}.
Can $\SW^{\Pin}$ be identified with some kind of {\it real } Gromov--Witten invariant?
{\it Cf.} Tian-Wang \cite{TW}.
\item What is the counter part of $\Pin^-(2)$-monopole equations in Donaldson theory?
Is there a version of Witten's conjecture between  $\Pin^-(2)$-monopole theory and  Donaldson theory?
More concretely, is $\SW^{\Pin}$ equivalent to some kind of Donaldson invariants?
\end{itemize}

\end{document}